\documentclass[a4paper,11pt]{amsart}

\usepackage{url}
\usepackage{color}
\definecolor{darkgreen}{rgb}{0,0.5,0}
\usepackage[
        colorlinks, citecolor=darkgreen,
        backref,
        pdfauthor={Christian Curilla, J. Steffen Mueller}, 
        pdftitle={Fermat curves of odd squarefree exponent}
]{hyperref}

\usepackage{amssymb,amsmath, amsthm}
\usepackage{comment}
\usepackage{algorithmic,algorithm}
\usepackage[OT2,OT1]{fontenc}
\usepackage{graphicx}
\usepackage{epsfig}
\usepackage{fancyhdr}
\usepackage{enumerate}
\usepackage{colonequals}
\usepackage{pdfsync}
\usepackage{xstring}

\usepackage[all]{xy}

\newcommand{\frakm}{\mathfrak{m}}

\newcommand{\frakp}{\mathfrak{p}}
\newcommand{\frakP}{\mathfrak{P}}
\newcommand{\frakq}{\mathfrak{q}}
\newcommand{\frakQ}{\mathfrak{Q}}

\newcommand{\frakF}{\mathfrak{F}}

\newcommand{\FF}{\mathbb{F}}
\newcommand{\NN}{\mathbb{N}}
\newcommand{\PP}{\mathbb{P}}
\newcommand{\QQ}{\mathbb{Q}}
\newcommand{\RR}{\mathbb{R}}
\newcommand{\ZZ}{\mathbb{Z}}

\newcommand{\can}{\pi}

\newcommand{\calC}{{\mathcal{C}}}
\newcommand{\calD}{{\mathcal{D}}}
\newcommand{\calE}{{\mathcal{E}}}
\newcommand{\calF}{{\mathcal{F}}}
\newcommand{\calG}{{\mathcal{G}}}

\newcommand{\calK}{{\mathcal{K}}}

\newcommand{\calO}{{\mathcal{O}}}

\newcommand{\calS}{{\mathcal{S}}}
\newcommand{\calT}{{\mathcal{T}}}

\newcommand{\calX}{{\mathcal{X}}}
\newcommand{\calY}{{\mathcal{Y}}}
\newcommand{\calV}{{\mathcal{V}}}
\newcommand{\calJ}{{\mathcal{J}}}
\newcommand{\calI}{{\mathcal{I}}}
\newcommand{\calZ}{{\mathcal{Z}}}

\newcommand{\fnpm}{\frakF_{N,\frakp}^{min}}
\newcommand{\fnm}{\frakF_{N}^{min}}
\newcommand{\om}{{\bar{\omega}}}
\newcommand{\omm}{{\bar{\omega}_{\fnm}}}
\newcommand{\spat}[1]{{{#1}^{\prime}}}

\newcommand{\Spec}{\operatorname{Spec}}
\newcommand{\Max}{\operatorname{Max}}
\newcommand{\Nm}{\operatorname{Nm}}

\newcommand{\Proj}{\operatorname{Proj}}
\newcommand{\Ht}{\operatorname{ht}}

\newcommand{\Pic}{\operatorname{Pic}}

\newcommand{\supp}{\operatorname{Supp}}

\newcommand{\Div}{\operatorname{Div}}
\renewcommand{\div}{\operatorname{div}}

\newcommand{\gr}[2]{\operatorname{gr}_{#1} (#2)}
\newcommand{\Sym}[1]{\operatorname{Sym} (#1)}
\newcommand{\frakM}{\mathfrak{M}}

\newcommand{\Frac}[1]{\operatorname{Frac}(#1)}

\numberwithin{equation}{section}

\newcommand{\belyi}{{\boldsymbol \beta}}

\newtheorem{thm}{Theorem}[section]
\newtheorem{coro}[thm]{Corollary}

\newtheorem{lemma}[thm]{Lemma}
\newtheorem{prop}[thm]{Proposition}
\newenvironment{prf}{\par\pagebreak[2]\noindent{\bf Proof: }}{
\hfill $\Box$ \bigskip}

\newenvironment{Prf}[1]{\par\pagebreak[2]\noindent{\bf Proof of #1: }}{
\hfill $\Box$ \bigskip}

\theoremstyle{definition}

\newtheorem{rem}[thm]{Remark}

\newtheorem{nota}[thm]{Notation}

\newtheorem{defin}[thm]{Definition}

\begin{document}

\title[Fermat curves of odd squarefree exponent]{The minimal regular model  
of a Fermat curve of odd squarefree exponent and its dualizing sheaf}

\author{Christian Curilla}
\address{Fachbereich Mathematik, Bereich AZ, Universit\"{a}t Hamburg, Bundesstrasse 55,
20146 Hamburg, Germany}
\email{c.curilla@web.de}
\author{J. Steffen M\"uller}
\address{Institut f\"ur Mathematik, Carl von Ossietzky
Universit\"{a}t Oldenburg, 26111 Oldenburg, Germany}
\email{jan.steffen.mueller@uni-oldenburg.de }

\date{\today}

\begin{abstract}
 We construct the minimal regular model of the Fermat curve of odd squarefree
 composite exponent $N$ over the $N$-th cyclotomic integers. As an application, we compute upper and lower bounds for 
 the arithmetic self-intersection of the dualizing sheaf of this model.
\end{abstract}

\maketitle

\begingroup
\hypersetup{linkcolor=black}
\tableofcontents
\endgroup

\section{Introduction}\label{Intro}
In the history of number theory and arithmetic geometry, the study of the 
Fermat curve
\begin{equation}\label{f_n}
  F_N : X^N+Y^N = Z^N
\end{equation}
of exponent $N\ge 3$ has played a prominent part.
In this paper we consider the case of the Fermat curve $F_N$ where $N$ is squarefree, odd and composite.

For explicit computations and bounds in the arithmetic geometry of curves over number fields,
one often needs to compute a regular model of the curve over the ring of integers.
While it is sometimes possible to compute a regular model of a given curve $X$ using, for instance,
the computer algebra system {\tt Magma}, the construction of regular models depending on a
parameter is more involved. 
In the case of the Fermat curve $F_p/\QQ(\zeta_p)$ of prime exponent $N=p\ge 3$ over the
field of $p$-th cyclotomic numbers, the minimal regular model 
 $\frakF^{min}_p$ over $\ZZ[\zeta_p]$ was constructed by McCallum~\cite{Mc}.
For other values of $N$, the minimal regular model
$\frakF^{min}_N$ of $F_N$ over $\ZZ[\zeta_N]$ is not available in the literature.

In Part~I of the present paper, we construct $\frakF^{min}_N$ when $N$ is squarefree, odd
and composite by following the construction of $\frakF^{min}_p$ due to
McCallum. However, the non-prime case is much more complicated.
It turns out that the only reducible fibers of $\frakF^{min}_N$ lie above primes
of $\ZZ[\zeta_N]$ dividing $N$, see Proposition~\ref{prop:goodprimes}.
For such a prime $\frakp$, the Zariski closure 
$\mathfrak{F}_{N,\frakp}^0$ of $F_N$ in $\PP^2_R$ consists of a single component of
multiplicity $p$, where $p$ is the residue characteristic and $R$ is the localization of
$\ZZ[\zeta_N]$ with respect to $\frakp$. Blowing up along this component, we obtain a
normal model.
The nonregular points of the latter can then be resolved by blow-ups, leading to a
regular model of $F_N \times_{\ZZ[\zeta_N]} R$.
The configuration of its special fiber is described in Theorem~\ref{thm:F_Nmin}, which
shows, in particular, that the model is minimal.
The local regular models can then be glued to construct the minimal regular model 
$\frakF^{min}_p$ .
Note that we can recover McCallum's results as a special case of our construction, see
Remark~\ref{sqfpr}.

Once an explicit description of $\frakF^{min}_N$ is available, several interesting
arithmetic invariants of $F_N$ can be computed, or at least bounded. These include some of
the invariants appearing in the conjecture of Birch and Swinnerton-Dyer, and
Arakelov-theoretic invariants. In Part~II of the present article, we consider the latter,
focusing on explicit bounds for the 
arithmetic self-intersection $\overline{\omega}_{\frakF^{\min}}^2$ of the relative dualizing
sheaf of $\frakF^{\min}$, equipped with the Arakelov metric. The computation of such
bounds was proposed in~\cite[p.~130]{Lang} and~\cite[\S8.2]{MB2}.

If $\calX$ is an arithmetic surface defined over the ring of integers $\calO_K$ of a
number field $K$ such that the generic fiber $X$ of $\calX$ has genus $g \ge 2$,
then the arithmetic self-intersection $\overline{\omega}_{\calX}^2$ of the relative dualizing
sheaf of $\calX$, equipped with the Arakelov metric, is one of the most important
invariants of $\calX$ (or, if $\calX$ is the minimal regular model of $X$, of $X$).
It is related to the Faltings height of $X$ and several other invariants,
see~\cite{javanpeykar} for a summary.
Lower bounds for $\overline{\omega}_{\calX}^2$ are crucial in the context of
the Bogomolov conjecture for curves, proved by Szpiro~\cite{Szpiro},
Zhang~\cite{Zhang1} and Ullmo~\cite{Ullmo}. However, an effective version of the
Bogomolov conjecture, which in the function field case is known due to work of
Zhang~\cite{Zhang2} and Cinkir~\cite{Cinkir}, is still an open problem in the number field
case.

On the other hand, suitable upper bounds for $\overline{\omega}_{\calX}^2$ in certain complete
families would lead to a proof of the effective Mordell conjecture,
see~\cite{Parshin,Vojta,MB2}.
Unfortunately, such bounds seem out of reach.
We summarize the known results in this direction.
Javanpeykar~\cite{javanpeykar} has given polynomial upper bounds in terms of the Belyi degree of $X$. 
While no bounds in complete families are known to date, there are some results for
discrete families. Namely, for certain positive integers $N$, there are bounds
for some modular curves, e.g $X_0(N)$, $ X_1(N)$ or $X(N)$, see~\cite{A.U, M.U, omega,
CurillaThesis, Mayer}.
Upper bounds for minimal regular models of Fermat curves $F_p$ of prime exponent $p$ over $\QQ(\zeta_p)$, where $\zeta_p$
is a primitive $p$-th root of unity, were first computed in~\cite{omega} and
vastly improved in~\cite{CK}.
They were complemented by lower bounds in~\cite[\S6]{KuehnMueller}.

Building on our explicit description of $\frakF^{min}_N$ from Part~I of this work, we use a result due to
K\"uhn~\cite{omega}, which can be viewed as an Arakelov-theoretic Hurwitz formula on
arithmetic surfaces, to compute upper bounds for $\omm^2$, when $N$ is odd, squarefree and
composite.
This is similar to the strategy used in the case of prime exponents~\cite{CK}.
We deduce the following result from the more precise Theorem~\ref{thm:fermat}:
\begin{thm}\label{thm:intro_upper}
  Let $N>0$ be an odd squarefree integer with at least two prime factors, and let $\mathfrak{F}_{N}^{min}$ be the minimal regular model of the Fermat curve
  $F_N$ over $\ZZ[\zeta_N]$. Then the arithmetic self-intersection number
  of its dualizing sheaf over $\ZZ[\zeta_N]$, equipped with the Arakelov metric, satisfies
\begin{equation}\label{eq:GrossON}
 \overline{\omega}_{\mathfrak{F}_N^{min} }^2 \leq (2g-2)\kappa\varphi(N)\log N +
 \calO(g\varphi(N)\log\log N)
\end{equation} 
  where
 $g=(N-1)(N-2)/2$ is the genus of $F_N$  
 and $\kappa\in \RR$ is a positive constant independent of $N$.
 \end{thm}
 In other words, Theorem~\ref{thm:intro_upper} yields an upper bound of order
 $N^2\varphi(N)\log N$.
To complement Theorem~\ref{thm:intro_upper}, we also compute a lower bound for $\omm^2$
using the results of~\cite{KuehnMueller}. These were already employed
in~\cite{KuehnMueller} in the case of prime exponents.
The following explicit lower bound follows from Theorem~\ref{thm:lower}:
\begin{thm}\label{thm:intro_lower}
In the notation of Theorem \ref{thm:intro_upper} we have the lower bound
\[
  \omm^2 > \frac{1}{5N^2}\varphi(N)\log(N)\,.
\]
\end{thm}
Although the results we obtain in Part~II are Arakelov-theoretic, we treat the results
from~\cite{omega} and~\cite{KuehnMueller} as black boxes. This reduces the computation of
our bounds to explicit computations of finite vertical intersection multiplicities on
$\frakF^{\min}_N$.

The paper is organized as follows: 
In Part~I, we first recall some preliminary results from algebraic
geometry in Section~\ref{sec:prelims}. 
These results are then used in Section~\ref{sec:curilla} to construct the local minimal
regular model $\fnpm$ of $F_N$ at a prime $\frakp$ of $\ZZ[\zeta_N]$ dividing $N$.
We switch to a global perspective in Section~\ref{sec:global} and construct the global
minimal regular model $\fnm$ of $F_N$ over $\ZZ[\zeta_N]$.

Part~II starts with a brief introduction to the arithmetic self intersection of the
relative dualizing sheaf on an arithmetic surface and how to compute lower and upper
bounds on it, see Section~\ref{sec:bound_om}.
In Section~\ref{sec:local_comps} we again work over a fixed prime $\frakp$ dividing $N$;
there we first study the extension of cusps of $F_N$ with respect to the Belyi morphism
$\belyi:F_N\rightarrow \mathbb{P}^1$ given by $(X:Y:Z)\mapsto (X^N:Y^N)$. After that, we define certain vertical
$\QQ$-divisors on the local minimal regular model $\fnpm$  and study their intersection
properties.
Finally we prove Theorem~\ref{thm:intro_upper} and Theorem~\ref{thm:intro_lower} in
Section~\ref{sec:upper}.
The proofs crucially rely on the local results of Section~\ref{sec:int}.

The results of Sections~\ref{sec:prelims},~\ref{sec:curilla}, \ref{sec:global}, and
of~\S\ref{sec:cusps} and~\S\ref{sec:upper_fermat} also appear in the first author's PhD thesis~\cite{CurillaThesis}, though the presentation has been
shortened and some of the proofs given there are different from those presented here.

\thanks{
We would like to thank Ulf K\"uhn for suggesting the work described in the present paper
and for answering many questions along the way.
We are also grateful to Vincenz Busch, Ariyan Javanpeykar, Franz Kir\'aly and Stefan Wewers for helpful discussions. 
}

\section*{\large Part~I: The minimal regular model of Fermat curves of odd squarefree
exponent}
  
\section{Preliminaries}\label{sec:prelims}
In the first two paragraphs we state a few results about regularity of Noetherian schemes
and about explicit blow-ups. 
These will be used in Section~\ref{sec:curilla} to construct the minimal regular
  model of the Fermat curve of odd squarefree exponent $N$ over $\ZZ[\zeta_N]$. 
Although most of the results are well-known, some of the statements or proofs seem to be
not easily accessible in the literature. 
  We hope that it will be useful for the other applications to have these tools gathered
in one place.
The final paragraph contains relevant definitions and results on arithmetic surfaces.

\subsection{Regularity}
We first develop some tools that help to decide whether a given scheme or ring is regular.

Let $A$ be a Noetherian local ring with maximal ideal $\frakm$ and residue class field
$k(\frakm)$. Recall that $A$ is \emph{regular} if $\dim A=\dim_{k(\frakm)} \frakm /\frakm^2$. 
Alternatively, $A$ is regular if and only if $\frakm$ can be generated by $\dim A$ elements.

More generally, let $A$ be a Noetherian ring.
If $\frakp \subset A$ is a prime ideal, then we say that $A$ is \emph{regular at} $\frakp$ if
the localization $A_\frakp$ is a regular local ring. We say that $A$ is \emph{regular} if it is regular at each prime ideal. 

\begin{lemma}\label{lemma:Reg1}
Let $A$ be a Noetherian ring and $\frakp \subset A$ a prime ideal. Then $A$ is regular at $\frakp$ if and only if $\frakp A_\frakp$ is generated by $\Ht (\frakp)$ elements.
\end{lemma}
\begin{prf}
This is obvious, since $\Ht (\frakp )=\dim A_\frakp$.
\end{prf}

\begin{lemma}\label{lemma:Reg3}
Let $A$ be a regular Noetherian ring and $S$ a multiplicative subset of $A$. Then $A_S$ is regular.
\end{lemma}
\begin{prf}
Let $\frakP$ be a prime ideal of $A_S$. This ideal is of the form $\frakp A_S$, where
$\frakp$ is a prime ideal of $A$ disjoint from $S$, see e.g. \cite[Theorem 4.1]{Mat}.
We have $(A_S)_{\frakp A_S}=A_\frakp$ by \cite[Corollary~4.4]{Mat}, hence the regularity of $A_S$ at $\frakP$ follows from the regularity of $A$ at $\frakp$.
\end{prf}

\begin{lemma}\label{lemma:Reg4}
Let $A$ be a Noetherian ring. Then $A$ is regular if and only if it is regular at its maximal ideals.
\end{lemma}
\begin{prf}
Follows from \cite[Corollary~4.4]{Mat}.
\end{prf}

In Section~\ref{sec:curilla} we have to check the regularity of a factor ring $A/f$, where $A$ is a regular ring and $f$ is an element of $A$. 
\begin{lemma}\label{lemma:sing}
Let $A/f$ be a factor ring, where $A$ is a regular ring and $f$ is an element of $A$.
Furthermore, let  $\frakP$ be a prime ideal of $A/f$ and $ \frakp =\can^{-1} \frakP$,
where $\can : A \rightarrow A/f$ is the canonical surjection. Then $A/f$ is regular at $\frakP$ if and only if $f \not\in (\frakp A_\frakp)^2$.
\end{lemma}
\begin{prf}
  The statement follows from \cite[Corollary~4.2.12]{Liu} and \cite[Theorem~4.2]{Mat}.
\end{prf}

Let $X$ be a locally Noetherian scheme and $x\in X$ a point. We say that $X$ is \emph{regular at} $x$\index{regular!scheme!at a point} if the stalk $\calO_{X,x}$ at $x$ of the structure sheaf $\calO_X$ is a regular local ring. We say that $X$ is \emph{regular}  \index{regular!scheme}if it is regular at all of its points. If $x$ is a point of $X$ which is not regular we call it a \emph{singular point of}\index{singular!point of a scheme} $X$. A scheme that is not regular is said to be \emph{singular}\index{singular!scheme}.

When our scheme comes with a flat morphism we can use the following useful result:
\begin{lemma}\label{lemma:Reg2}
Let $X$ and $Y$ be locally Noetherian schemes and $g:X\rightarrow Y$ a flat morphism. If $Y$ is regular at $y\in g(X)$, and $X_y=X\times_Y \Spec k(y)$ is regular at a point $x$, then $X$ is regular at $x$.
\end{lemma}
\begin{prf}
See~\cite[Corollaire 6.5.2]{ega4_2}.
\end{prf}

In the situations we consider later the scheme $Y$ is already regular and we only need to
take care of the scheme $X_y$. This scheme is a variety over the field $k(y)$. To analyze
the points of this variety we can use the Jacobian criterion~\cite[Theorem 2.19]{Liu}.

\begin{rem}\label{rk:ffreg}
Let us assume the morphism $g$ in Lemma~\ref{lemma:Reg2} is \emph{faithfully flat}, i.e.
flat and surjective. If $Y$ and $X_y$ are regular for all $y\in Y$ then $X$ is regular. If
$X$ is regular then $Y$ is regular by \cite[Corollaire~6.5.2]{ega4_2}. If $Y$ is regular
at $y$ and $X_y$ is singular at some $x$ it may still happen that $X$ is regular at $x$. 
\end{rem}

Now we are going to describe how we can use regularity to show normality.

\begin{prop}\label{prop:normal}
Let $R$ be a regular integral Noetherian ring and $f \in R \setminus R^{\ast}$. If $R/f$
is regular in codimension 1, then $R/f$ is normal.
\end{prop}
\begin{prf}
Since $R$ is a regular ring, it is a Cohen-Macaulay ring. We want to show that $R/f$ is a
Cohen-Macaulay ring as well.
Let $\frakm\in\Max\left( R/f \right)$ and $\frakM \in \Max\left(R\right)$ be the preimage
of $\frakm$. Since localization commutes with passing to quotients
by ideals, we have \[ \left(R/f\right)_{\frakm}=R_{\frakM}/fR_{\frakM} \, .\]  Now $f$ is
a regular element of $R_{\frakM}$ and so  $R_{\frakM}/fR_{\frakM}$ is a Cohen-Macaulay
ring (see \cite[Proposition~8.2.15]{Liu}. Because our computation is valid for
all maximal ideals of $R/f$, the ring $R/f$ is Cohen-Macaulay, cf.~\cite[Proposition
18.8]{EisenCA}. The statement follows using Serre's criterion, see for instance
\cite[Theorem~8.2.23]{Liu}.
 \end{prf}

\subsection{Blow-ups}

In the study of birational morphisms blow-ups play an important role. 
We summarize the main facts we need about them.  Most of the material we introduce
is standard and the proofs may be found, for instance, in \cite{Liu} and \cite{EisHar}.  Later
we will prove a result which deals with the concrete situation that we will encounter 
in Section~\ref{sec:curilla}. Apart from this we mostly follow Liu's book~\cite{Liu}.

 To start with, let $A$ be a Noetherian ring and $I$ an ideal of $A$. We denote by
 $\widetilde{A}$ the graded $A$-algebra \[ \widetilde{A}=\bigoplus_{d \geq 0} I^{d} ,
 \mbox{    where  } I^0\colonequals A \, . \]
 \begin{defin} \index{blow-up! of an affine scheme} Let $X=\Spec A$ be an affine
   Noetherian scheme, $I$ an ideal of $A$, and $\widetilde{X}=\Proj \widetilde{A}$. The
   scheme $\widetilde{X}$ together with the canonical morphism $\widetilde{X} \rightarrow
   X$ is called the \emph{blow-up of $X$ along $V(I)$}. 
 \end{defin}
 The blow-up has the following properties.
 \begin{lemma}\label{lem:reg1}
 Let $A$ be a Noetherian ring, and let $I$ be an ideal of $A$. 
 \begin{enumerate}
 \item The ring $\widetilde{A}$ is integral if and only if $A$ is integral.
 \item Let $B$ be a flat $A$-algebra, and let $\widetilde{B}$ be the graded $B$-algebra associated to the ideal $IB$. Then we have a canonical isomorphism $\widetilde{B} \cong B \otimes_A \widetilde{A}$.
 \end{enumerate}
 \end{lemma}
 \begin{prf}
   See \cite[Lemma~8.1.2.]{Liu}.
 \end{prf}
 
Now let $I=(a_1, \ldots, a_r)$. We denote by $t_i\in I =\widetilde{A}_1$ the element
$a_i$,
considered as a homogeneous element of degree 1. We have a surjective homomorphism of
graded $A$-algebras \[ \phi : A[X_1,\ldots, X_r] \rightarrow \widetilde{A} \]  defined by
$\phi (X_i)=t_i$. It follows that $\widetilde{A}$ is isomorphic to a factor ring
$A[X_1,\ldots, X_r]/J$; here $J$ denotes an ideal of $A[X_1,\ldots, X_r]$. It may be
desirable for certain applications to express the blow-up in such a way. Unfortunately it
is not always easy to describe the ideal $J$ explicitly. However, if the ideal $I$ is
generated by a regular sequence, we have a simple description of $J$.

\begin{lemma}\label{lem:reg2}
Let $I\subset A$ be an ideal which is generated by a regular sequence $a_1, \ldots, a_r$. Then $\widetilde{A}\cong A[X_1, \ldots , X_r]/J$ where the ideal $J$ is generated by the elements of the form $X_ia_j-X_ja_i$ for $1 \leq i,j \leq r$.
\end{lemma}
\begin{prf}
  See~\cite[Proposition~IV-25, Exercise~IV-26]{EisHar}.
\end{prf}
 
Later on, we will mostly work with integral rings. Here we have the following situation:
 \begin{lemma}\label{lem:reg4}
 Let $A$ be a Noetherian integral ring and $I=(a_1, \ldots, a_r)$ an ideal of $A$ such
 that $a_i \neq 0$ for all $i$. The blow-up $\widetilde{X}\rightarrow X=\Spec A$ along $V(I)$ is the union of the affine open subschemes $\Spec A_i$, $1\leq i\leq r$, where $A_i$ is the sub-$A$-algebra \[ A[\frac{a_1}{a_i}, \ldots, \frac{a_r}{a_i}] \] of the field $\Frac{A}$ generated by the $\frac{a_j}{a_i}\in \Frac{A}$, $1\leq j \leq r$.
 \end{lemma}
 \begin{prf}
   See for instance \cite[Lemma~8.1.4]{Liu}.
 \end{prf} 
 
 \begin{lemma}\label{lem:reg3}
Let $A$ be an integral Noetherian ring, $a_1,\ldots, a_r$ a regular sequence, and $I=(a_1,\ldots,a_r)$. We have: \begin{enumerate}
  \item The ring \[ R=A[X_1,\ldots, \widehat{X_i}, \dots, X_r] / J  \] is integral, where $J$ is
   generated by the elements $a_j-X_j a_i$ with $1\leq j\leq r$ and $j\neq i$.
 \item For an element $f\in A$ let $\overline{f}$ denote its image in $R$. We have \[ f\in I^d \Leftrightarrow \overline{f} \in (\overline{a_i})^d \, . \]
 \end{enumerate}
 \end{lemma}
 \begin{prf}
Since $A$ is integral, $\widetilde{A}$ is integral as well by Lemma~\ref{lem:reg1}. We
know that \[\widetilde{A}\cong A[X_1, \ldots, X_r]/ J\, , \] where $J$ is generated by the
elements $X_ia_j-X_ja_i$ for $1\leq i,j, \leq r$, see Lemma \ref{lem:reg2}. Hence $\Spec R$ is an affine open subset of $\Proj \widetilde{A}$ and therefore integral. This proves the first statement. 

For the second statement we assume $i=1$ for simplicity. Let $f \in I^d$. Then there
exists a homogeneous polynomial $F(X)=F(X_1, \ldots, X_r) \in A[X_1, \ldots, X_r]$ of
degree $d$ such that $f=F(a)=F(a_1, \ldots, a_r)$. If we set  \[ f_0=\frac{F(a_1, X_2a_1,
\ldots, X_ra_1)}{a_1^d}=F(1, X_2, \ldots, X_r)\, , \] then we obviously have $\overline{f}=\overline{f_0}\overline{a_1}^d$ and therefore $\overline{f} \in (\overline{a_1})^d$. \\
Now let $\overline{f} \in (\overline{a_1})^d$. Furthermore, let $n$ be the largest integer
such that $f\in I^n$. Let us assume $n<d$. As above, there is a homogeneous polynomial
$F(X)$ of degree $n$ with $F(a)=f$. It follows that not all coefficients of $F(X)$ are in
$I$ because otherwise we would have $f\in I^{n+1}$. Now $ f_0=\frac{F(a_1,X_2a_1, \ldots,
X_ra_1)}{a_1^n} $ is a polynomial in $X_2, \ldots , X_r$ whose coefficients are not all
in $I$. We have $\overline{f}=\overline{f_0}\overline{a_1}^n$, but,
since  $R$ is integral and $\overline{f} \in (\overline{a_1})^d$ with $n<d$, the element $\overline{a_1}$ must divide $\overline{f_0}$. Therefore 
$f_0=a_1 G(X)+H(X)$, where $G(X)\in A[X_2,\ldots, X_r]$ and $H(X)\in  J$. It follows that all coefficients of $f_0$ are in $I$, a contradiction. In other words, we have $d\leq n$ and therefore $f \in I^d$.
\end{prf}
 
So far we have discussed the blow-up of an integral scheme along a subscheme associated to an ideal generated by a regular sequence. 
Unfortunately, we will encounter more involved blow-ups in Section~\ref{sec:curilla}.
However, in those situations the following theorem will come to our aid.

\begin{thm}\label{thm:reg}
Let $A$ be an integral Noetherian ring, $a_1,\ldots, a_r$ a regular sequence, and
$I=(a_1,\ldots,a_r)$ a prime ideal of $A$. Furthermore, let $f\in I$ and $n$ be the
largest integer such that $f \in I^n$. Then \[ A[X_1,\ldots,\widehat{X_i}, \ldots , X_r]
/J_0\cong A/f[\frac{\boldsymbol{a}_1}{\boldsymbol{a}_i},\ldots,
\frac{\boldsymbol{a}_r}{\boldsymbol{a}_i}]\, , \] where $J_0$ is the ideal generated by the
$a_j-X_j a_i$ (with $1\leq j \leq r$ and $j\neq i$) and a polynomial $f_0$ such that
$f\equiv f_0 a_i^n \bmod J$; here $\boldsymbol{a}_j$ denotes the residue class of $a_j$ in $A/f$ and $J$ is the ideal from Lemma \ref{lem:reg3}.
 \end{thm}   
 \begin{prf}
 For simplicity we assume $i=1$. The canonical surjection  
  \begin{align*}
 \varphi : A[X_2, \ldots, X_r] \longrightarrow  & \,   A/f[\frac{\boldsymbol{a}_2}{\boldsymbol{a}_1}, \ldots,\frac{\boldsymbol{a}_r}{\boldsymbol{a}_1}]  \\
F(X_2, \ldots, X_r) \longmapsto & \boldsymbol{F}(\frac{\boldsymbol{a}_2}{\boldsymbol{a}_1}, \ldots,\frac{\boldsymbol{a}_r}{\boldsymbol{a}_1}) 
\end{align*} 
(here the bold $\boldsymbol{F}$ indicates that we reduce the coefficients of the polynomial modulo $f$) induces, since $ a_i-X_i a_1 \in\ker \varphi$, a surjection \begin{align*}
\phi : A[X_2, \ldots, X_r]/J \longrightarrow & \, A/f[\frac{\boldsymbol{a}_2}{\boldsymbol{a}_1}, \ldots,\frac{\boldsymbol{a}_r}{\boldsymbol{a}_1}] \\
\overline{F}(\overline{X_2}, \ldots \overline{X_r}) \longmapsto & \,
  \boldsymbol{F}(\frac{\boldsymbol{a}_2}{\boldsymbol{a}_1},
  \ldots,\frac{\boldsymbol{a}_r}{\boldsymbol{a}_1})  \, ,
\end{align*} where $J$ is the ideal from Lemma \ref{lem:reg3}. We get the following commutative diagram \begin{equation}\label{dia} \xymatrix{
A[X_2, \ldots, X_r]/J \ar@{->>}[r]^{\phi}& A/f[\frac{\boldsymbol{a}_2}{\boldsymbol{a}_1}, \ldots,\frac{\boldsymbol{a}_r}{\boldsymbol{a}_1}]\\
         A \ar@{->}[u] \ar@{->}[r]^{\pi} & A/f \ar@{_{(}->}[u] 
} \end{equation} Next we want to investigate the kernel of the map $\phi$. Let
$x=\overline{F}(\overline{X}_2, \ldots \overline{X_r})$, where $F(X_2, \ldots,
X_r)$ is a polynomial of degree $m$ and $\phi (x)=0$. We have $a_1^m F(X_2, \ldots, X_r) \equiv \mu \mod
J$, where $\mu \in A$. Since diagram \eqref{dia} is commutative and the right
arrow in this diagram is injective, we have $\can (\mu)=0$. It follows that $\mu=\lambda f$
for some $\lambda \in A$. Now let $n$ ($n_\lambda$ resp.) be the largest integer such that $f\in
I^{n}$ ($\lambda \in I^{n_\lambda}$ resp.) and $f_0 \in A[X_2, \ldots, X_r]$ ($\lambda_0
\in A[X_2, \ldots, X_r]$ resp.) with $\overline{a_1}^{n}\overline{f_0}=\overline{f}$
($\overline{a_1}^{n_\lambda} \overline{\lambda_0}=\overline{\lambda} $ resp.). We have
\begin{equation}\label{eq:kongru}\overline{a_1}^m x=\overline{f}\overline{\lambda}=
  \overline{a_1}^{n}\overline{f_0} \overline{a_1}^{n_\lambda} \overline{\lambda_0}
\end{equation} in $A[X_2, \ldots, X_r]/J$. If we assume that $m\leq n+n_\lambda$, then we can
cancel $\overline{a_1}^m$ in equation \eqref{eq:kongru} by Lemma~\ref{lem:reg3} and it
follows that $x$ is in the ideal $(\overline{f_0})$. So if we can show that
$m>n+n_\lambda$ is impossible, then we are done. According to \eqref{eq:kongru}
we have $\lambda f \in I^m$ by Lemma~\ref{lem:reg3}. Now $m>n+n_\lambda$ would imply that
the associated graded algebra $\gr{I}{A}$ is not integral. But $a_1, \ldots, a_r$ is a
regular sequence and so we have an $A/I$-algebra isomorphism \[ \Sym{I/I^2}\cong \gr{I}{A}
\] (see \cite{Hun1}) where $\Sym{I/I^2}$ is integral, because $I$ is a prime ideal. This finishes
the proof by contradiction.
\end{prf}

\begin{rem}\label{rem:blow-up}
The schemes we have to consider later are of the form $\Spec A/f$ (at least locally),
where $A$ is a ring and $f \in A$ is a prime element.
The blow-up of $A/f$ along $V(I/f)$ is covered by the spectra of the  rings \[
A/f[\frac{\boldsymbol{a}_1}{\boldsymbol{a}_i}, \ldots,
\frac{\boldsymbol{a}_r}{\boldsymbol{a}_i} ]\, , \] where $\boldsymbol{a}_j$ is the residue
class of $a_j$ in $A/f$ and $I=(a_1,\ldots,a_r)$, cf. Lemma~\ref{lem:reg4}. According to Theorem~\ref{thm:reg} we can
express these rings explicitly as factor rings\, if the $a_j$ form a regular sequence and $I$ is a prime ideal. To do this, we only need to know
the largest integer $n$ such that $f\in I^n$ and polynomials $f_{0,i}$ such that $f \equiv
f_{0,i} a_i^n\mod J$. We can use the following strategy to find these quantities: We only
need to find a homogeneous polynomial $F(X)\in A[X_1, \ldots, X_r]$ such that not all
coefficients are in $I$ and such that $F(a)=f$. Obviously $f\in I^n$, where $n$ is the degree of
$F(X)$. Because $a_1, \ldots, a_r$ is a regular sequence, it is a quasi-regular sequence as
well, see \cite[Theorem 16.2]{Mat}. It follows that if $f\in I^{n+1}$, then all
coefficients of $F(X)$ are in $I$, a contradiction. So $n$ is the largest integer such that $f\in
I^n$. We can compute the $f_{0,i}$ as in the proof of Lemma~\ref{lem:reg3}.
More precisely, we have \[ f_{0,i}=F(X_1, \ldots, X_{i-1}, 1, X_{i+1}, \ldots, X_r) \, .\]\end{rem}

We briefly describe how to extend the construction of blow-ups of affine scheme to arbitrary schemes.
In this situation we need to use a coherent sheaf of ideals to construct the blow-up.
\begin{defin}
Let $X$ be a Noetherian scheme, and $\calI$ be a coherent sheaf of ideals on $X$. Consider
the sheaf of graded algebras $\bigoplus_{d \geq 0} \calI^d$, where $\calI^d$ is the $d$-th
power of the ideal $\calI$, and set $\calI^0=\calO_X$. Then $\widetilde{X}=\Proj
\bigoplus_{d \geq 0} \calI^d$ is the \emph{blow-up of $X$ with respect to $\calI$}. If $Y$ is the closed subscheme of $X$ corresponding to $\calI$, then we
also call $\widetilde{X}$ the \emph{blow-up of $X$ along $Y$}.
\end{defin}  

\begin{prop}\label{prop:blow1}
Let $X$ be a locally Noetherian scheme, and let $\calI$ be a coherent sheaf of ideals on
$X$. Let $\pi : \widetilde{X} \rightarrow X $ be the blow-up of $X$ along $Y=V(\calI)$.
Then we have the following properties: \begin{enumerate}
\item The morphism $\pi$ is proper.
\item Let $Z \rightarrow X$ be a flat morphism with $Z$ locally Noetherian. Let
  $\widetilde{Z}\rightarrow Z$ be the blow-up of $Z$ along $\calI \calO_Z$; then $\widetilde{Z}\cong \widetilde{X}\times_X Z $.
\item The morphism $\pi$ induces an isomorphism $\pi^{-1} (X \setminus V(\calI))\rightarrow X\setminus V(\calI)$. If $X$ is integral, and if $\calI \neq 0$ , then $\widetilde{X}$ is integral, and $\pi$ is a birational morphism.
\end{enumerate}
\end{prop}   
\begin{prf}
  See for instance \cite[Proposition~8.1.12]{Liu}.
\end{prf}

Now let us assume that $X$ is a locally Noetherian scheme that comes with a
closed immersion $f:X\rightarrow Z$ to a locally Noetherian scheme $Z$. Let $\calJ$ be a
quasi-coherent sheaf of ideals on $Z$ with the property that $f(X)$ is not contained in
the center $V(\calJ)$. Then the blow-up $\widetilde{X}$ of $X$ along $\calI$, where
$\calI=(f^{-1} \calJ)\calO_X$, is a closed immersion of the blow-up $\widetilde{Z}$ of $Z$
along $\calJ$, see for instance \cite[Corollary 1.16]{Liu}. The closed subscheme
$\widetilde{X} \subseteq \widetilde{Z}$ is called the \emph{strict transform}
\index{strict transform} of $X$. In our applications the scheme $X$ will be a singular scheme which is a subscheme of a
regular scheme $Z$. We will use a sequence of blow-ups of $X$ to compute a desingularization
of this scheme. Each of these blow-ups comes from a blow-up of the scheme $Z$. The
blow-ups of $Z$ are regular by~\cite[Theorem~8.1.19]{Liu}.

\subsection{Intersection theory on arithmetic surfaces}\label{sec:intas}
Let $R$ be a Dedekind ring with field of fractions $K$.
If $\pi:\calX \to \Spec R$  is a projective flat morphism and $\calX$  a regular integral scheme of dimension $2$ such that 
the generic fiber
\[ \calX_K = \calX\times_{\Spec R} \Spec K \] of $\pi$ is geometrically
irreducible, we call $\calX$ an \emph{arithmetic surface}.
If $X/K$ is a geometrically irreducible smooth projective curve and $\calX$ is an
arithmetic surface over $R$ whose generic fiber $\calX_K$ is isomorphic to $X$, then we
call $\calX$ a \emph{(projective) regular model} of $X$ over $R$.
Such a model always exists, see for instance~\cite{LipmanDesing}. 
Moreover, if the genus of
$X$ is at least~1, then there always exists a regular model $\calX^{min}$ of $X$ over $R$, unique up to
isomorphism, such that every $R$-birational morphism $\calX^{min} \rightarrow \calX$ to another regular model
$\calX$ of $X$ over $R$ is an isomorphism. 
We call $\calX^{min}$ the \emph{minimal regular model} of $X$ over $R$.
A regular model $\calX$ of $X$ over $R$ is minimal if and only if none of its irreducible
components can be contracted by a blow-up morphism such that the resulting model remains
regular; such components are called \emph{exceptional}. If $\calC$ is a component of a
special fiber $\calX_s$ that is defined over an algebraically closed field, then,
by Castel\-nuovo's criterion~\cite[Theorem~9.3.8]{Liu}, 
$\calC$ is exceptional if and only if it has genus~0 and self-intersection $-1$, see below. 

Let $\pi:\calX \to \Spec R$ be an arithmetic surface.
If $s \in \Spec R$ is a closed point and $D, E$ are divisors on $\calX$ without common
component, we denote by 
$(D\cdot E)_s$ the rational-valued intersection multiplicity between $D$ and $E$
(cf.~\cite[\S9.1.2]{Liu}); we simply write
$(D\cdot E)$ if it is clear which $s$ we are working over.
If $D$ is a vertical divisor on $\calX$ with support in the fiber $\calX_s$, then we can use the moving
lemma~\cite[Corollary~9.1.10]{Liu} to define the self-intersection $D_s^2$ (or $D^2$).
We extend the intersection multiplicity $(\;\cdot\;)$ to the group
\[
  \Div(\calX)_\QQ \colonequals  \Div(\calX) \otimes_\ZZ \QQ
\]
of \emph{$\QQ$-divisors on $\calX$} by linearity.

Let $\omega_{\calX/R}$ denote the \emph{relative dualizing sheaf} of $\calX$ over $R$. 
We call a divisor $\calK$ of $\calX$ such that $\calO_{\calX} (\calK) \cong \omega_{\calX
/R}$ a \emph{canonical divisor}. 
More generally,  we call a divisor $\calK\in \Div(\calX)_\QQ$ such that $\calO_{\calX}
(\calK) \cong
\omega_{\calX /R}$ in $\Pic (\calX)\otimes_\ZZ\QQ$ a \emph{canonical $\QQ$-divisor}.
If $\calE$ is an effective nonzero vertical divisor, we define 
\begin{equation}\label{a_c}
  a_\calE := \calE^2- 2p_a(\calE)+ 2\,.
\end{equation}
where $p_a(\calE)$ is the arithmetic genus of $\calE$.
\begin{thm}[Adjunction formula]\label{thm:kan}
Let $\calK$ be a canonical $\QQ$-divisor on $\calX$ and let $\calE\ne 0$ be an
effective vertical divisor on $\calX$. Then we have
\begin{equation}\label{adform} a_\calE = (\calK\cdot \calE)\,.\end{equation}
\end{thm}
\begin{prf}
See ~\cite[Theorem~8.1.37]{Liu} for the case $\calK \in \Div (\calX)$. 
The extension to  $\calK \in \Div (\calX)_\QQ$ is immediate.
\end{prf}

We will use the adjunction formula extensively, especially in Section~\ref{sec:int}.

\section{The local minimal regular model}\label{sec:curilla}

Let $N$ be an odd squarefree natural number which is not prime and let
$\zeta_N$ be a primitive $N$-th root of unity. 
Recall that the Fermat curve $F_N/\QQ(\zeta_N)$ is defined by \begin{equation*} F_N :
X^N+Y^N=Z^N\,. \end{equation*}
Let $p$ be a prime number such that $N=pm$ with $m \in \NN$ and fix a prime ideal $\frakp $ of $\ZZ[\zeta_N]$ that
lies above $p$.
We denote by $R$ the localization of $\ZZ[\zeta_N]$ with respect to $\frakp$.
In this section we construct the minimal regular model of $F_N \times_{\Spec\ZZ[\zeta_N]}
\Spec R$, see Theorem~\ref{thm:F_Nmin}. 

Let $\pi$ be a uniformizing element of $R$ and let $k(\pi)$ denote its residue field, viewed as
a subfield of $\overline{\FF}_p$. We can and will also interpret this element as a
uniformizing element of the strict Henselization $R^{sh}$. Consider the model
\[\mathfrak{F}_{N,\frakp}^0= \Proj R[X,Y,Z]/(X^N+Y^N-Z^N) \, .\]
To construct the minimal regular model of $F_N \times_{\Spec\ZZ[\zeta_N]}\Spec R$ 
we work with affine open subschemes of $\mathfrak{F}_{N,\frakp}^0$. 
In particular, we consider the integral affine open subscheme
\begin{equation} \label{sch:F_N} \calX\colonequals  \Spec R[X,Y]/(X^N+Y^N-1) \end{equation}
of $\mathfrak{F}_{N,\frakp}^0$. 
For a natural number $n$ we will also use  $F_n$ to denote the polynomial
$X^n+Y^n-1$. It will be clear from the context whether we refer to the Fermat curve or to the
polynomial, by abuse of notation. For the following computations it will be useful to rewrite $X^N+Y^N-1$ as 
\begin{equation}\label{eq:F_mAlt} F_m^p+p\psi(X^m,Y^m)\, , \end{equation}
where \begin{equation}\label{psi} \psi(a,b)=\frac{a^p+b^p-1-(a+b-1)^p}{p} \, . \end{equation}
Note that there is a unit $\mu$ of $R$ such that $p=\mu \pi^{p-1}$.
Using~\eqref{eq:F_mAlt}, it can be seen easily that the special fiber
of $\calX$ is of the form \[ \Spec(R[X,Y]/(F_m^p+p\psi(X^m,Y^m))\otimes_R k(\pi))=\Spec
(k(\pi)[X,Y]/F_m^p) \, . \]  
Therefore the special fiber consists of a single component $\calC$, which has multiplicity $p$. 
This component -- considered as a subset of $\calX$ --  is
the closure of the ideal $I=(\pi,F_m) \subset R[X,Y]/(X^N+Y^N-1)$, so
$V(I)=\calC$. The ideal $I$ is a prime ideal, as the ring \begin{equation*} R[X,Y]/I
  \cong k(\pi)[X,Y]/(X^m+Y^m-1) \end{equation*} is integral. Because of the regularity of
  this ring, the closed subscheme $\calC$ is regular. However, since $F_N\in
  I^{p-1}$ and $p \neq 2$, the scheme $\calX$ is singular. In fact, it is not even normal,
  because it is not regular in codimension 1.

\subsection{The polynomial \texorpdfstring{$\psi(X^m,Y^m)$}{PsiXmYm}}\label{sec:poly}

In this  paragraph we are going to study the polynomial $\psi(X^m,Y^m)$, see~\eqref{psi}. 
In order to do this we analyze the polynomial $\psi(a,b)$ and then evaluate it in $X^m$
and $Y^m$ later on. We have the following: 
\jot3mm
 \begin{align*}
\psi(a,b)-\psi(a,1-a) &= \frac{a^p+b^p-1-(a+b-1)^p}{p}-\frac{a^p+(1-a)^p-1}{p}\\
                      &= \frac{b^p-{(a+b-1)}^p+(a-1)^p}{p}\\ &= \sum_{k=1}^{p-1} \frac{{p \choose k}}{p}{(a+b-1)}^{p-k}b^k(-1)^k \, .
\end{align*}
\jot1mm
Substituting $X^m$ for $a$ and $Y^m$ for $b$ we get \begin{equation}\label{gl2}
  \psi(X^m,Y^m)=\psi(X^m,1-X^m)+\sum_{k=1}^{p-1} \frac{{p \choose k}}{p}{F^{p-k}_m}Y^{mk}(-1)^k \end{equation}

For later computations it will be important to know the factorization of $\psi(X^m,Y^m)$
into irreducibles.  We first recall a result of McCallum~\cite{Mc}.

\begin{lemma}\label{lem:poly}
There is a decomposition  \begin{equation}\label{zer1}
  \psi(a,1-a)=a(a-1)\Psi (a)\, , \end{equation} with a polynomial $\Psi (a) \in R[a]$. In
  the prime factorization of $\Psi (a)$ over $\overline{\FF}_p $, factors occur with
  multiplicity one if they are not rational over $\FF_p$, and with multiplicity two otherwise. 
\end{lemma}
\begin{prf}
We elaborate on the proof of the Lemma on page 59 of~\cite{Mc}.
We have $(\psi(a,1-a))^{\prime}=a^{p-1}-(1-a)^{p-1}\equiv -(a-2)\cdot\ldots\cdot (a-p+1)
\bmod \pi$. The only roots of $\psi(a,1-a) \bmod \pi$ with multiplicity greater than one are
of the form $\overline{\alpha} \in \left\{\overline{2},\ldots,\overline{p-1}\right\}$ with
$\alpha \in R$. If the multiplicity of $\overline{\alpha}$ were greater
than two, then the second derivative would vanish in $\overline{\alpha}$ as well. But from
$(p-1)\alpha^{p-2}+(p-1)(1-\alpha)^{p-2}\equiv 0 \bmod \pi$ it follows that
$\alpha^{p-2}\equiv (\alpha-1)^{p-2} \bmod \pi$, so by multiplication with
$\alpha(\alpha-1)$ we obtain $\alpha-1 \equiv \alpha\bmod \pi$  and this is obviously impossible. Let us denote the root of multiplicity 2 by $\overline{\alpha}_1, \ldots, \overline{\alpha}_s$ 

Together with the fact that 0 and 1 are simple roots of $\psi(a,1-a)$ and $\overline{\psi}(a,1-a)$, we get the decomposition\begin{equation}\label{eq:decomp} \overline{\psi}(a,1-a)= a(a-1)(a-\overline{\beta}_1)\cdot\ldots\cdot (a-\overline{\beta}_r)(a-\overline{\alpha}_1)^2\cdot\ldots\cdot (a-\overline{\alpha}_s)^2  \, ,\end{equation} over $\overline{\FF}_p$, where $\overline{\beta}_i \notin \FF_p$.
 with some irreducible polynomials $\overline{f}_i(a)$. Since in this decomposition all
 factors are pairwise coprime and $\deg \psi(a,1-a))=\deg \overline{\psi} (a, 1-a)$, the claim follows from Hensel's lemma.
 \end{prf}

\begin{coro}\label{coro:poly}
There is a decomposition  \begin{equation}\label{zer2} \psi(X^m,1-X^m)=X^m
\prod_{i=0}^{m-1}(X-\zeta_m^i)\Psi (X^m)\, . \end{equation} In the prime factorization of $\Psi (X^m)$ over $\overline{\FF}_p $, factors $(X-\overline{\delta})$ occur with multiplicity 1 if $\overline{\delta}^m$ is not rational over $\FF_p$, and with multiplicity 2 otherwise. 
\end{coro}
\begin{prf}
If we replace $a$ by $X^m$ in \eqref{zer1}, it is obvious that we get \eqref{zer2}, since $\zeta_m^i \in R$. A decomposition as in \eqref{eq:decomp} becomes \[
\overline{\psi} (X^m, 1-X^m)=X^m
\prod_{i=0}^{m-1}(X-\overline{\zeta}_m^i)(X-\overline{\delta}_1)\cdot \ldots \cdot
(X-\overline{\delta}_{rm})(X-\overline{\gamma}_1)^2 \cdot \ldots \cdot
(X-\overline{\gamma}_{sm})^2 \] after this substitution; here
$\overline{\delta}^m=\overline{\beta}$ and $\overline{\gamma}^m=\overline{\alpha}$. Since
the $\overline{\alpha}_i$ and $\overline{\beta}_j$ from Lemma \ref{lem:poly} are non-zero,
the polynomials $X^m-\overline{\alpha}_i$ and $X^m-\overline{\beta}_j$ split into
coprime linear factors over $\overline{\FF}_p$. The linear polynomials
$(X-\overline{\gamma}_k)$ are the only factors of multiplicity two in $\Psi (X^m)$ over $\overline{\FF}_p$.
\end{prf}

\begin{defin}\label{def:scrW}
Let us denote by $\varrho$ the number of factors $(X-\overline{\gamma}_k)^2$ of $\Psi
(X^m, 1-X^m)$ over $\overline{\FF}_p$.
\end{defin}

\begin{rem}
As $\psi (a, 1-a)$ is a polynomial of degree $p-1$, the polynomial $\psi (X^m, 1-X^m)$ is
of degree $m(p-1)$. Corollary \ref{coro:poly} tells us that there are \[\deg \Psi (X^m) -
2\varrho=m(p-3)-2\varrho\] linear factors of multiplicity one in $\Psi (X^m)$. For
instance, let $p=5$. Then $\Psi_5(a)\equiv a^2-a+1 \mod 5$, where $a^2-a+1$ is an irreducible element of $\FF_5 [a]$. 
It follows that in this case $\varrho=0$. On the other hand, consider the case $p=7$. Here we have $\Psi_7 (a) \equiv (a+2)^2 (a+4)^2 \mod 7$, hence $\varrho=\frac{1}{2}\deg \Psi_7 (X^m)=2m $.
\end{rem}

\subsection{The blow-up of \texorpdfstring{$\calX$}{X} along \texorpdfstring{$V(I)$}{V(I)}}\label{subsec:blow}

We start by giving an explicit description of the blow-up.
\newcommand{\blow}[1]{{U_{#1}}}
\newcommand{\lblow}[1]{{U_{1, #1}}}
\newcommand{\lblowprime}[1]{{U^{\prime}_{1, #1}}}
\newcommand{\lblowtil}[1]{{\widetilde{U}_{#1}}}
\newcommand{\bW}[1]{{W_{#1}}}
\newcommand{\bw}[1]{{w_{#1}}}

\begin{prop}\label{prop:blowup1}
Let $I$ denote the ideal $I=(\pi, F_m) \subset R[X,Y]/F_N$.
Then the blow-up $\widetilde{\calX}$ of the scheme $\calX$ in \eqref{sch:F_N} along $V(I)$ is given by the affine open subsets $U_1=\Spec S_1$ and $U_2=\Spec S_2$, where \begin{equation}\label{eq:card1} S_1=R[X,Y,\bW{1}]/(F_m-\bW{1}\pi,\pi \bW{1}^p+\mu\psi(X^m,Y^m)) \end{equation} and \begin{equation}\label{eq:card2}S_2=R[X,Y,\bW{2}]/(\bW{2} F_m-\pi,F_m+\mu \bW{2}^{p-1}\psi(X^m,Y^m)) \, .\end{equation} In other words, we have $\widetilde{\calX}=\blow{1} \cup \blow{2}$.
\end{prop}
\begin{prf}
The generators of the ideal $I$ obviously form a regular sequence in $R[X,Y]$, since
$R[X,Y]$ and $R[X,Y]/\pi$ ($R[X,Y]/F_m$ resp.) are integral. Therefore we can apply Theorem \ref{thm:reg}. The polynomial \[ F_m\bW{1}^{p-1}+\mu \bW{2}^{p-1}\psi (X^m,Y^m)\in \left(R[X,Y]\right)[\bW{1},\bW{2}] \] is homogeneous in $\bW{1}$ and $\bW{2}$ and the coefficient $\mu \psi(X^m,Y^m)$ is not in the ideal $I$. The statement follows now with Remark \ref{rem:blow-up}.
\end{prf}

\begin{rem}\label{rem:STreg}
The scheme $\widetilde{\calX}$ can be considered as a subscheme of the scheme
$\widetilde{\calZ}=V_1 \cup V_2$, where   \begin{equation*} V_1= \Spec R[X,Y,\bW{1}]/(F_m-\bW{1}\pi)
\end{equation*} and \begin{equation*}V_2= \Spec R[X,Y,\bW{2}]/(\bW{2} F_m-\pi) \, .\end{equation*}
Since $\widetilde{\calZ}$ is just the blow-up of the regular scheme $\calZ=\Spec
R[X,Y]$ along $(\pi, F_m)$, it is regular as well by
\cite[Lemma~8.1.4]{Liu} and~\cite[Theorem~8.1.19]{Liu}. The scheme $\widetilde{\calX}$ is the strict transform of $\calX$ in $\widetilde{\calZ}$.
\end{rem}

\begin{prop}\label{prop:normalF_N}
The scheme $\widetilde{\calX}$ from Proposition \ref{prop:blowup1} is normal. Let
$\overline{F}_m,\, \overline{\psi}(X^m,1-X^m)\in \overline{\FF}_p[X,Y] $ be the respective
reductions of $F_m$ and $\psi(X^m,1-X^m)$ with respect to the canonical morphism
$R[X,Y]\rightarrow \overline{\FF}_p [X,Y]$. The geometric special fiber $\widetilde{\calX}
\times_{\Spec R} \Spec \overline{\FF}_p$ has configuration as in Figure \ref{fig:normal},
where the components $L_{(x, y)}$ are of genus $0$ and are parameterized by those pairs
$(x,y)\in\overline{\FF}_p^2$ which satisfy \[x^m+y^m-1=\overline{\psi}(x^m,1-x^m) = 0 \, .\]

\begin{figure}[h]\centering
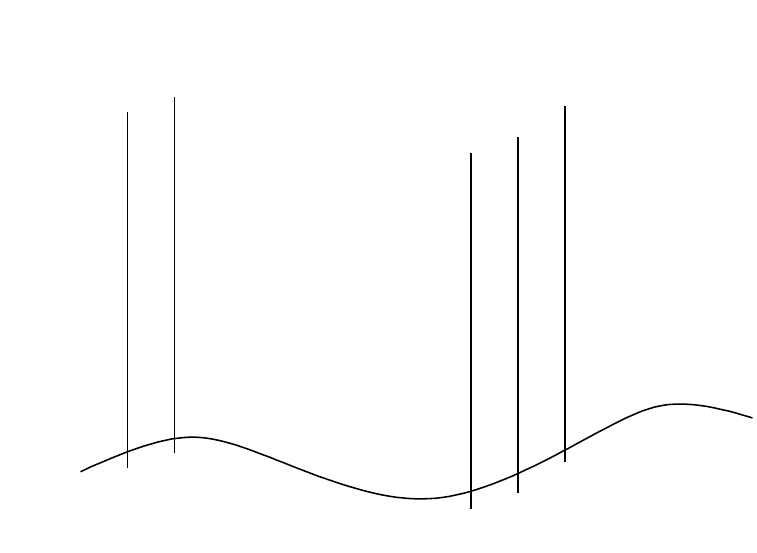
\parbox{12cm}{\caption{The configuration of the geometric special fiber $\widetilde{\calX} \times_{\Spec R} \Spec \overline{\FF}_p$.}\label{fig:normal}}
\end{figure}
\end{prop}
\begin{prf}
We work with the scheme
\begin{equation}\label{eq:13}\widetilde{\calX}^{sh}=\widetilde{\calX}\times_{\Spec R}
  \Spec R^{sh}\end{equation} whose special fiber is a variety over the algebraically
  closed field $\overline{\FF}_p$. Since this base change is faithfully flat, normality of
  $\widetilde{\calX}^{sh}$ implies normality of $\widetilde{\calX}$ . We start our computation with the affine open subscheme $\blow{1}^{sh}=\Spec S_1^{sh}$, where $S_1^{sh}=S_1\otimes_R R^{sh}$. The special fiber of this scheme is  
\jot3mm
\begin{align}
\blow{1}^{sh}\times_{\Spec R^{sh}} \Spec \overline{\FF}_p &= \Spec \left( \overline{\FF}_p [X,Y,\bW{1}]/ (F_m, \psi(X^m,Y^m))\right) \notag \\
&= \Spec \left( \overline{\FF}_p [X,Y,\bW{1}]/ (F_m, \psi(X^m,1-X^m))\right) \, .\label{eq:19}
\end{align}
\jot1mm
 This variety consists of lines $L_{x, y}=V(X-x, Y-y)$, where $x$ is a root of
 $\overline{\psi}(X^m,1-X^m)$ and $y$ is a root of $Y^m+x^m-1\in \overline{\FF}_p[Y]$.
 These lines correspond to prime divisors $V(\frakP)$ of $\blow{1}^{sh}$, where
 $\frakP=(X-\spat{X},Y-\spat{Y}, \pi)$ is a prime ideal of height 1 and $\spat{X}\equiv x
 \mod \pi$ ($\spat{Y}\equiv y \mod \pi$ resp.). Because of Remark \ref{rem:STreg} and
 Proposition \ref{prop:normal}, it suffices to show that $S_1^{sh}$ is regular
 at $\frakP$ (since the generic fiber of $\widetilde{\calX}^{sh}$ ($\blow{1}^{sh}$ resp.) is
 regular, $S_1^{sh}$ is regular at every prime ideal which does not contain $\pi$). Note
 that $\pi$ cannot be a divisor of $\spat{X}$ and of $\spat{Y}$, as $x^m+y^m=1$. Because
 of symmetry, we may assume $\pi \nmid \spat{Y}$ without loss of generality.  We have
 $\psi(\spat{X}^m, 1-\spat{X}^m)=\lambda \pi$, where $\lambda \in R^{sh}$. Now,\[
 \psi(X^m, 1-X^m)=\lambda \pi +(X-\spat{X})G(X)\, ,\] where $G(X) \in R^{sh}[X]$. 
It follows from Proposition \ref{prop:blowup1} and equation \eqref{gl2} that \[
-(X-\spat{X})G(X)=\pi \left(\bW{1}^{p}\mu^{-1}+ \bW{1}Y^{m(p-1)}+\lambda +\pi H(Y,\bW{1}) \right)\]  in
$S_1^{sh}$, where $H(Y,\bW{1})\in R^{sh}[Y,\bW{1}]$.

 Let us suppose that $\bW{1}^{p}\mu^{-1}+ \bW{1}Y^{m(p-1)}+\lambda+\pi H(Y,\bW{1}) \in \frakP$. Then
 $\bW{1}^{p}\mu^{-1}+ \bW{1} \spat{Y}^{m(p-1)}+\lambda\in \frakP$ and (using Hensel's lemma) we have
 $(\bW{1}-\spat{\bW{}})\in \frakP$, where $\spat{\bW{}}$ is a root of $\bW{1}^{p}\mu^{-1}+ \bW{1}
 \spat{Y}^{m(p-1)}+\lambda=:f(\bW{1})\in R^{sh}[\bW{1}]$. Indeed, since
 $\overline{f^{\prime}}(\bW{1})=y^{m(p-1)}\neq 0$ the polynomial $\overline{f}(\bW{1})$ splits into
 coprime linear factors in $\overline{\FF}_p$, and this decomposition lifts to $R^{sh}$.
 But if this linear factor is in $\frakP$, then $\frakP$ is a maximal ideal; a
 contradiction, because $\frakP$ was assumed to be of height 1. Hence we have
 \[\bW{1}^{p}\mu^{-1}+ \bW{1}Y^{m(p-1)}+\lambda+\pi H(Y,\bW{1}) \notin \frakP\, ,\] and so this element becomes a unit in $(S_1^{sh})_\frakP$. We denote this unit by $\epsilon$.\\
Note that, since $\pi | \spat{X}^m+\spat{Y}^m-1$, we have $\spat{X}^m+\spat{Y}^m-1=\tau \pi$,
where $\tau \in R^{sh}$. Using Proposition \ref{prop:blowup1}, it follows that
\jot3mm
 \begin{align*} 
\pi \bW{1} &= X^m+Y^m-1 \\
&= X^m-\spat{X}^m+Y^m-\spat{Y}^m+\spat{X}^m+\spat{Y}^m-1 \\
&= (X-\spat{X})\prod_{i=1}^{m-1}(X-\spat{X}\zeta_m^i)+(Y-\spat{Y})\prod_{i=1}^{m-1}(Y-\spat{Y}\zeta_m^i)+\tau\pi 
\end{align*}
\jot1mm
 in $S_1^{sh}$. Now, $\prod_{i=1}^{m-1}(Y-\spat{Y}\zeta_m^i)\notin \frakP$ because
 otherwise $\spat{Y}\in \frakP$ or $(1-\zeta_m^i) \in \frakP$ and this is impossible,
 since these elements are units in $R^{sh}$. To see this, recall that $\pi \nmid
 \spat{Y}$, and that $(1-\zeta_m^i)$ is a divisor of $m$ and $m$ is coprime to $p$.
 Therefore $\prod_{i=1}^{m-1}(Y-\spat{Y}\zeta_m^i)$ is a unit in $(S_1^{sh})_\frakP$. We will denote this unit by $\epsilon^{\prime}$. In the localization $(S_1^{sh})_\frakP$ we have 
 \[-(X-\spat{X})G(X)\frac{1}{\epsilon}=\pi \]
 and \[ -(X-\spat{X})\left(
 \prod_{i=1}^{m-1}(X-\spat{X}\zeta_m^i)+G(X)\frac{1}{\epsilon}(\bW{1}-\tau)\right)\frac{1}{\epsilon^{\prime}}=(Y-\spat{Y})
 \, .\] Hence we have $\frakP (S_1^{sh})_\frakP=(X-\spat{X})$ and so $S_1^{sh}$ is
 regular at $\frakP$  by Lemma~\ref{lemma:Reg1}. 
 
 We still have to deal with the second affine open subscheme $\blow{2}^{sh}=\Spec S_2^{sh}$, where
 $S_2^{sh}=S_2\otimes_R R^{sh}$. It suffices to check the regularity of $S_2^{sh}$ at the
 prime ideal \begin{equation}\label{eq:21} \frakP=(\bW{2},F_m, \pi)\, ,\end{equation} which
 corresponds to the component $F_m$ in Figure \ref{fig:normal}. But in $S_2^{sh}$ we even
 have $\frakP=(\bW{2})$ by Proposition \ref{prop:blowup1}, and so this ring is obviously regular at $\frakP$.
 \end{prf}

 \subsection{Resolving the singularities of $\widetilde{\calX}$}
We now find the singular closed points of the normal scheme $\widetilde{\calX}$
and then resolve these singularities. We shall see that for the resolution it sufficed to
blow up the lines that have singular points lying on them. Since blowing up commutes with flat morphisms, we can work with $\widetilde{\calX}^{sh}$ instead of $\widetilde{\calX}$
throughout, as long as we only blow up along ideal sheaves $\calJ$ of
$\widetilde{\calX}^{sh}$ which are of the form $\calI \calO_{\widetilde{\calX}^{sh}}$,
where $\calI$ is an ideal sheaf of $\widetilde{\calX}$. Before we come to the main result
of this section we need to introduce some further terminology.
We continue to use the notation of Proposition \ref{prop:normalF_N}.

\begin{defin}
We call a component $L_{(x,y)}$ of $\widetilde{\calX}^{sh}=\widetilde{\calX}\times_{\Spec
R} \Spec R^{sh}$ \emph{a component of type $A$}\index{component of $\widetilde{\calX}$!of
type $A$}, if $x=0$ or $x^m=1$, and \emph{a component of type $B$}\index{component of
$\widetilde{\calX}$!of type $B$}, if $x$ is a multiple root of $\overline{\psi}(X^m,1-X^m)$ different from $0$.
\end{defin}

We first find and resolve the singularities on $\calX^{sh}$.
In the following, we call a curve of genus~0 over $\overline{\FF_p}$ a {\em line}.

\begin{thm}\label{thm:F_Nreg}
Let $\widetilde{\calX}^{sh}$ be the normal scheme given by \eqref{eq:13}. If we blow up
$(m-1)$-times along the components of type $A$, we get $p$ chains consisting of $(m-1)$
lines (Figure \ref{fig:thmbild3}). Blowing up along the components of type $B$ gives $p$
chains consisting of one line (Figure \ref{fig:thmbild2}). The resulting scheme is regular.
\begin{figure}[h!]\centering
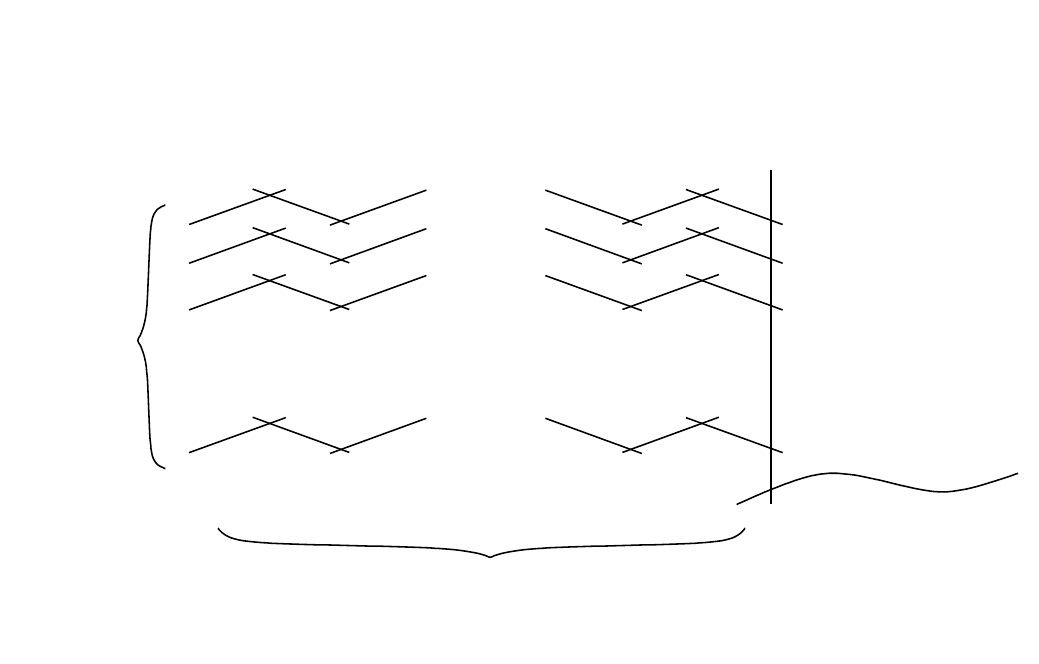
\parbox{12cm}{\caption{The configuration of the components after $(m-1)$-times blowing up
a component $L_{(x_a,y_a)}$ of type $A$.   }\label{fig:thmbild3}}
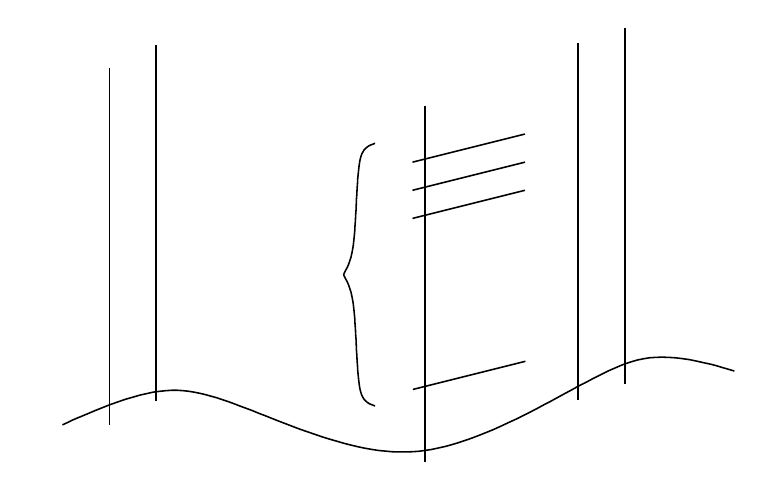
\parbox{12cm}{\caption{The configuration of the components after blowing up a component
$L_{(x_b,y_b)}$ of type $B$.  }\label{fig:thmbild2}}
\end{figure}
\end{thm}

For the proof of the theorem we first need three preparatory lemmata.

\begin{lemma}\label{lem:part1}
In the notation of Proposition \ref{prop:normalF_N}, the only singular points of
$\widetilde{\calX}^{sh}$ lie on the components $L_{(x,y)}$ of type $A$ and of type $B$ (Figure \ref{fig:thmbild1}). 
\end{lemma} 

\begin{prf}
\begin{figure}[h!]\centering
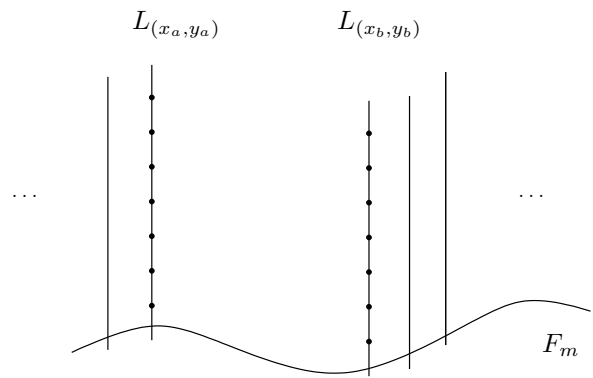
\parbox{12cm}{\caption{The line $L_{(x_a,y_a)}$ is of type $A$ and the line
$L_{(x_b,y_b)}$ is of type $B$.}\label{fig:thmbild1}}
\end{figure}
We first use the Jacobian criterion to locate the singular points on the affine open
subset \[\blow{1}^{sh}\times_{\Spec R^{sh}} \Spec \overline{\FF}_p =\Spec \left(
\overline{\FF}_p [X,Y,\bW{1}] / (F_m, \psi(X^m,1-X^m))\right) \,,  \] see~\eqref{eq:19}.
The Jacobian matrix is of the form \[ J(X,Y,\bW{1})=\left( \begin{array}{ccc}mX^{m-1} &mY^{m-1}&0 \\  
G^{\prime}(X) & 0 & 0 \end{array}\right) \, ,\] where $G(X)=\overline{\psi}(X^m, 1-X^m)$.
It follows that a closed point $P=(x,y,\bw{})\in U_1\times_{\Spec R} \Spec \overline{\FF}_p$ is
singular if and only if \begin{equation*}-my^{m-1}G^{\prime}(x)=0 \, .\end{equation*} Now
$y=0$ implies $x^m-1=0$, and so $x$ is an $m$-th root of unity. In case $G^{\prime}(x)=0$,
the element $x$ is an $m$-th root of an element of $\FF_p^{\ast}$ or $0$ by Corollary~\ref{coro:poly}. 

Note that $F_m$ is the only component of the special fiber of $\widetilde{\calX}^{sh}$ which does not lie in
$\blow{1}^{sh}$. To find its singular points, we work on the affine open subset $\blow{2}^{sh}$.
A closed point which lies on $F_m$ corresponds to a maximal ideal \[ \frakm=(\pi, \bW{2},
X-\spat{X},Y-\spat{Y})\subset S_2^{sh}\, ,\]  where $\spat{X}^m+\spat{Y}^m \equiv 1 \mod
\pi$, cf.~\eqref{eq:21}. Without loss of generality we may again assume $\pi \nmid
\spat{Y}$.
Using arguments similar to those in the proof of Proposition \ref{prop:normalF_N}
combined with \eqref{eq:card2}, we see that
in $S_2^{sh}$ we have \[ (Y-\spat{Y})\epsilon^{\prime} \in (\pi,\bW{2},X-\spat{X})
\subset S_2^{sh} \, ,\] where $\epsilon^{\prime}=\prod_{i=1}^{m-1}(Y-\spat{Y}\zeta_m^i)
\notin \frakm$. Together with the fact that $\pi=\bW{2}F_m$ in
$S_2^{sh}$, this gives us \[ \frakm (S_2^{sh})_\frakm = (\bW{2}, X-\spat{X})\ ;\] hence $S_2^{sh}$
is regular at $\frakm$ by Lemma~\ref{lemma:Reg1}. Therefore there are no singular points
lying on components which are not of type $A$ or of type $B$. 
\end{prf} 

Lemma \ref{lem:part1} shows us that we have to focus on the components of type $A$ and of
type $B$. Let us analyze the former. A component $L_{(x_a,y_a)}$ of type $A$ corresponds
to a prime ideal \[ \frakP=(\pi, X, Y-\zeta_m^i)\subset S_1^{sh} \, .\]  There is an
affine open neighborhood $\blow{}$ of $\frakP$ with the property that $V(\frakP) \subset \blow{}=\Spec
A \subseteq \blow{1}^{sh}$ and $\frakP A=(\pi, X)$. To be more precise, we have $ Y^m -1=
(Y-\zeta_m^i)f$, where $f$ is the product of the $(Y-\zeta_m^j)$ with $j \neq i$. Then we
may take $A$ to be \begin{equation}\label{eq:14} A=S/(\pi \bW{1}^p + \mu \psi (X^m, Y^m))\,
  ,\end{equation} where \[ S=\left( R^{sh}[X,Y,\bW{1}]/(F_m-\bW{1} \pi)\right)_f \] is the
  localization of $R^{sh}[X,Y,\bW{1}/(F_m-\bW{1} \pi)$ with respect to the set $\{ 1, f, f^2, f^3,
  \ldots \}$. Hence $\blow{}$ is isomorphic to the principal open subset $D(f)$ of $\blow{1}^{sh}$.
  Note that, as $\frakP$ is a regular prime ideal of height one, it is possible to find an
  affine open neighborhood $\blow{}^{\prime}$ so that $\frakP$ is generated by one element in this neighborhood. Unfortunately $\blow{}^{\prime}$ does not contain $V(\frakP)$.

Next, we study schemes which naturally appear as blow-ups of the scheme $\Spec A$.

\begin{lemma}\label{lem:part2}
 Let $l\in \NN$ with $1\leq l \leq m-1$ and \begin{equation}\label{eq:A_l} A_l\colonequals
   S[T_l]/(\pi-T_l X^l, g_l (T_l)) \, , \end{equation} where
   \begin{equation}\label{eq:T_l} g_l(T_l)=T_l \bW{1}^p + \mu\frac{\psi(X^m,
     1-X^m)}{X^l}+\mu\sum_{k=1}^{p-1} {p \choose k}p^{-1}{(T_l
     \bW{1})}^{p-k}X^{l(p-k-1)}Y^{mk}(-1)^k \, .\end{equation} Furthermore, let $\lblow{l}=\Spec
     A_l$. Then $\lblow{l}$ is normal;  the configuration of the special fiber of $\lblow{l}$ is given
     in Figure \ref{fig:A_l}. The only components of the special fiber which correspond to
     prime ideals that contain $X$ are given by $L_{l,1}, \ldots , L_{l,p}$ and
     $L_{(x_a,y_a)}$. If $l=m-1$, there are no singular closed points lying on these
     components. If $l<m-1$, the only singular closed points are the points where the
     components $L_{l,i}$ intersect the component $L_{(x_a,y_a)}$.
  \begin{figure}[t]\centering
 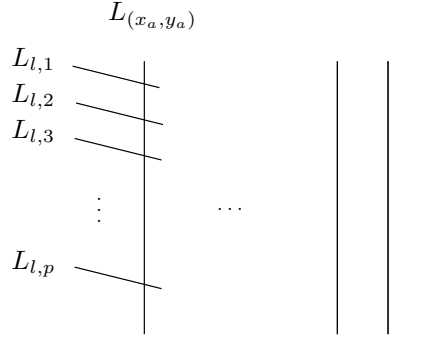
 \parbox{12cm}{\caption{The configuration of the special fiber of $\lblow{l}$.}
 \label{fig:A_l}} \end{figure}
\end{lemma}
\begin{prf}
First of all note that $\lblow{l}$ is a closed subscheme of the regular integral scheme
$V_l=\Spec S[T_l]/(\pi-T_l X^l)$. To see that $V_l$ is integral and regular one may
observe that even the ring \[ B=R^{sh}[X,Y,\bW{1},T_l]/(F_m -\bW{1}\pi, \pi -T_l X^l) \] has these
properties: We have that $\pi, X^l$ is a regular sequence in the integral ring
$R^{sh}[X,Y,\bW{1}]/(F_m-\bW{1} \pi)$, so the ring $B$ is one of the rings we get if we blow up
$R^{sh}[X,Y,\bW{1}]/(F_m-\bW{1} \pi)$ along the ideal $(\pi, X^l)$, see Lemma \ref{lem:reg3}. It
follows that $B$ is integral  by Lemma \ref{lem:reg1}. To see the regularity we use the
Jacobian criterion and find that the only maximal ideals which can be singular are of the
form \[ \frakm=(\pi, X, Y-\zeta_m^i, T-\spat{T}, \bW{1}-\spat{\bW{}})\, ,\] where $\spat{T}, \spat{\bW{}}
\in R^{sh}$ and $i\in \ZZ$. We have the chain of prime ideals \[0 \subsetneq (\pi,X,
Y-\zeta_m^i) \subsetneq (\pi, X, Y-\zeta_m^i, T-\spat{T}) \subsetneq \frakm\, .\] On the other
hand, $\frakm B_\frakm=(X, T-\spat{T}, \bW{1}-\spat{\bW{}})$. This gives us $3 \leq \dim B_\frakm
\leq \dim_{k(\frakm)} \frakm/ \frakm^2 \leq 3$, hence the regularity of $B_\frakm$. It
follows from Lemma~\ref{lemma:Reg4} that $B$ is regular. 

 Let us return to the scheme $\lblow{l}$ and show that it is normal. In order to do this we may
 first consider the affine open subscheme $\lblowprime{l}=\Spec (A_l)_X$, where $(A_l)_X$ is
 the localization of $A_l$ with respect to the set \[ \{1, X, X^2,X^3, \ldots\} \, .\] The
 special fibers of $\lblowprime{l}$ and of $\lblow{l}$ have the same configuration, except that
 $\lblowprime{l}$ does not include components corresponding to prime ideals that contain $X$
 and $\pi$. An easy computation shows that $(A_l)_X \cong (S^{sh}_1)_{Xf}=(S_1 \otimes_R
 R^{sh})_{Xf}$ (cf.~\eqref{eq:card1}), where $Xf$ is the multiplicative subset $\{1, f, X,
 Xf, X^2, f^2, \ldots\}$. It follows that $\lblowprime{l}$ is normal and that its special
 fiber has the same configuration as the special fiber of $\blow{1}^{sh}=\Spec S^{sh}_1$ after
 removing the components $L_{(x,y)}$ with $x=0$, cf. Proposition \ref{prop:normalF_N}.

 Next, let us analyze the components of the special fiber of $\lblow{l}$ that do not lie in
 $\lblowprime{l}$. 
For a prime ideal $\frakP \subset A_l$ such that $\pi , X \in \frakP$ we have \begin{equation}\label{eq:A_lnormal} T_l \bW{1}^p+\mu T_l \bW{1} (\zeta_m^i)^{m(p-1)}=T_l \bW{1} (\bW{1}^{p-1}+\mu)\in \frakP \, , \end{equation} hence the only prime ideals of height one with this property are
\[(\pi, X, T_l),\;\;(\pi, X, \bW{1}),\;\;\textrm{and}\;\;(\pi, X, \bW{1}-\theta\zeta_{p-1}^i) \, ,
\]
 where $0\leq i \leq p-2$ and $\theta$ is an element of $R^{sh}$ satisfying $\theta^{p-1}=-\mu$.
 Note that $\frakP$ can only contain one of the elements $T_l$, $\bW{1}$ or
 $\bW{1}-\theta\zeta_{p-1}^i$, because otherwise $\frakP=A_l$ or $\frakP$ is a maximal ideal,
 hence it is of height 2. Since $\pi=T_l X^l$ in $A_l$ it follows from \eqref{eq:T_l} and
 \eqref{eq:A_lnormal} that $\frakP (A_l)_{\frakP}=(X)$, and therefore that $\lblow{l}$ is normal. \\
 Let $\frakm=(X, T_l-\spat{T}, \bW{1}-\spat{\bW{}})$ be a maximal ideal of $A_l$ such that $\pi \nmid
 \spat{T}$ (note that $\pi \in \frakm$ since $\pi=T_l X^l$ in $A_l$). It follows from
 \eqref{eq:T_l} and \eqref{eq:A_lnormal} that $\spat{T}
 \bW{1} (\bW{1}^{p-1}+\mu) \in \frakm$ and so
 we may assume without loss of generality that $\spat{\bW{}}=0$ or $\spat{\bW{}}= \theta
 \zeta_{p-1}^i$. Since the factors \begin{equation}\label{eq:factors} \bW{1}, (\bW{1}-\theta),
 (\bW{1}-\theta\zeta_{p-1}), (\bW{1}-\theta\zeta_{p-1}^2), \ldots, (\bW{1}-\theta\zeta_{p-1}^{p-2})
 \end{equation} are pairwise coprime, \eqref{eq:T_l} and~\eqref{eq:A_lnormal} show us that
 $(\bW{1}-\spat{\bW{}})$ is contained in the ideal of $(A_l)_\frakm$ which is generated by $X$ and
 $(T-\spat{T})$. Hence the ring $A_l$ is regular at $\frakm$. Next, let
 $\frakm=(X,T_l,\bW{1}-\spat{\bW{}})$, where $(\bW{1}-\spat{\bW{}})$ is coprime to all of the factors in \eqref{eq:factors}. Then $\bW{1}(\bW{1}^{p-1}+\mu)$ becomes a unit in the localization with respect to $\frakm$. Again,~\eqref{eq:T_l} and~\eqref{eq:A_lnormal} yield $\frakm (A_l)_\frakm=(X,\bW{1}-\spat{\bW{}} )$ and therefore the regularity of $A_l$ at $\frakm$. 
 
Finally, we consider the case $\frakm=(X, T_l, \bW{1}-\spat{\bW{}})$, where $\spat{\bW{}}=0$ or
$\spat{\bW{}}=\theta\zeta_{p-1}^i$  for some integer $0 \le i \le p-2$. We may distinguish here between two cases. In case $l=m-1$, we have
  \begin{equation}\label{eq:12} -T_{(m-1)} \bW{1} (\bW{1}^{p-1}+\mu)= \mu X \left( \frac{\psi(X^m,
    1-X^m)}{X^m} + P(T_{(m-1)})\right)\,  \end{equation} in $A_{(m-1)}$; here
    $P(T_{(m-1)})\in S[T_{(m-1)}]$ is the polynomial given by  \[
    P(T_{(m-1)})=\sum_{k=1}^{p-2} \frac{{p \choose k}}{  p}{(T_{(m-1)}
  \bW{1})}^{p-k}X^{{(m-1)}(p-k-1)-1}Y^{mk}(-1)^k \, . \] Obviously we have $P(T_{(m-1)})\in
  \frakm$. If the term in parentheses on the right-hand side of \eqref{eq:12} were contained
  in $\frakm$, then we would have \[ \frac{\psi(X^m, 1-X^m)}{X^m} \in \frakm \, ,\] a contradiction.
  Hence this term becomes a unit in $(A_{(m-1)})_\frakm$, and we have \[ \frakm
  (A_{(m-1)})_\frakm=(T_{(m-1)}, \bW{1}-\spat{\bW{}})\, .\] In other words, $A_{(m-1)}$ is regular
  at $\frakm$. \\
  Now consider the case $l < m-1$. Let $\frakM$ be the prime ideal of the
  regular ring $S[T_l]/(\pi-T_l X^l)$ which is given by the preimage of $\frakm$. Since
  $(Y-\zeta_m^i)=-(X^m-\bW{1} T_l X^l)f^{-1}$ in $S[T_l]/(\pi-T_l X^l)$, we have
  $(Y-\zeta_m^i)\in \frakM^2$, which yields \[ g_l (T_l) \equiv T_l \bW{1}^p + \mu T_l \bW{1} \equiv
  0 \mod \frakM^2 \, . \] Hence $A_l$ is singular at $\frakm$.
   Let us denote the components which correspond to the prime ideals $(\pi,X,\bW{1})$ and $(\pi,X,\bW{1}-\theta \zeta_{p-1}^i)$ for $0\leq i \leq p-2$ by $L_{l,1},\ldots, L_{l,p}$. The configuration of $\lblow{l}\times_{\Spec R^{sh}} \Spec \overline{\FF}_p$ is given in Figure \ref{fig:A_l}.  
 \end{prf}
 
  \begin{lemma}\label{lem:part3}
  We use the notation from Lemma \ref{lem:part2}. Let $l<m-1$. If we blow up along the
  ideal $(X,T_l)$ the resulting scheme is covered by the affine open subset $\lblow{l+1}$ (cf.
  Lemma \ref{lem:part2}) and an affine open subset $\lblowtil{l+1}=\Spec
  \widetilde{A}_{l+1}$. The configuration of the special fiber is given by Figure
  \ref{fig:A_l} (replacing $l$ by $l+1$) in $\lblow{l+1}$ and by Figure \ref{fig:tilde{A}_{l+1}} in $\lblowtil{l+1}$. The scheme $\lblowtil{l+1}$ is regular.
   \begin{figure}[t]\centering
 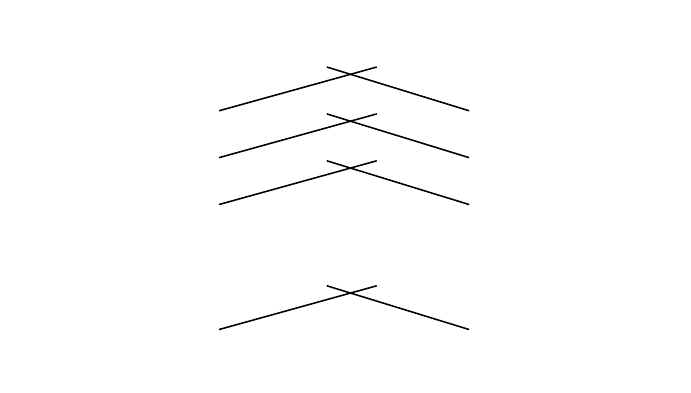
\parbox{12cm}{\caption{The configuration of $\Spec \widetilde{A}_{l+1}\times_{\Spec
R^{sh}}\Spec \overline{\FF}_p$.}
\label{fig:tilde{A}_{l+1}}}
\end{figure}
  \end{lemma}
  
\begin{prf}
We blow up along the ideal $(X,T_l)$.  Setting $\frac{X}{T_l}=\widetilde{X}$, one affine
open subset of the blow-up is isomorphic to $\Spec\widetilde{A}_{l+1}$, where \[
\widetilde{A}_{l+1} \colonequals  S[T_l,\widetilde{X}]/ (\pi-T_l^{l+1}\widetilde{X}^l, \widetilde{X} T_l- X, \widetilde{g}_l (\widetilde{X}))
\cong A_l \left[{X}{T_l}^{-1} \right] \, , \] and \begin{align*}
 \widetilde{g}_l (\widetilde{X})=  \bW{1}^p & + \mu\frac{\psi((\widetilde{X} T_l)^m,
 1-(\widetilde{X} T_l)^m)}{\widetilde{X}^{l} {T_l}^{l+1}} \\ & +\mu\sum_{k=1}^{p-1} {p \choose
 k}p^{-1}{T_l}^{(l+1)(p-k-1)}\widetilde{X}^{l(p-k-1)}\bW{1}^{p-k}Y^{mk}(-1)^k \, .\end{align*} 
 A prime ideal $I$ which contains $\pi$ also contains $X$ and $Y-\zeta_m^i$, since $T_l \in I$
 or $\widetilde{X} \in I$. Furthermore, in case $\widetilde{X} \in I$, we have $\bW{1}^p +
 \mu \bW{1} \in I$. Hence, the  prime ideals of height 1 which contain $\widetilde{X}$ are of
 the form $(\widetilde{X}, G(\bW{1}))$, where $G(\bW{1})$ is one of the factors in
 \eqref{eq:factors}. We denote these prime ideals by $\frakP_1, \ldots, \frakP_p$. In
 case $T_l\in I$ we have $\bW{1}^p + \mu \bW{1} \in I$ as well. 
 We denote the prime ideals $(T_l, G(\bW{1}))$ by $\frakQ_1, \ldots, \frakQ_p$. A maximal ideal
 $\frakm$ of $\widetilde{A}_{l+1}$ is of the form $\frakm=(\widetilde{X}, G(\bW{1}), T_l
 -\spat{T})$ ($\frakm=(T_l, G(\bW{1}), \widetilde{X}-\spat{X})$ resp.). If we localize with
 respect to this ideal, the corresponding ideal in the localization is generated by
 $\widetilde{X}$ and $T_l - \spat{T}$ ($T_l$ and $\widetilde{X}-\spat{X}$ resp.), hence
 the ring is regular at $\frakm$. Since these are the only maximal ideals of this ring,
 the ring itself is regular by Lemma~\ref{lemma:Reg4}. 
The blow-up-morphism $\lblowtil{l+1}=\Spec \widetilde{A}_{l+1} \rightarrow \Spec A_l$
is an isomorphism away from $V(X,T_l)$. The components
$L_{l,i}$ of $\lblow{l}$ are the images of the components which correspond to the prime ideals
$\frakP_i \subset \widetilde{A}_{l+1}$ Therefore we denote these components by
$L_{l,i}$ as well. The components which lie above the singular points are denoted by $L_{l+1,i}$. They correspond to the prime ideals $\frakQ_i$.
Then the special fiber has the configuration as in Figure \ref{fig:tilde{A}_{l+1}}. The
component $L_{l,i}$ intersects the component $L_{l+1,i}$ in the point corresponding to some $\frakm= (\widetilde{X}, T_l, G(\bW{1}))$.

Let us now take a look at the other affine open subset of the blow-up. Setting
$T_{l+1}=\frac{T_l}{X}$, we get \[ A_l\left[ {T_l}{X}^{-1} \right]\cong S[T_l, T_{l+1}]/(\pi -T_{l+1}X^{l+1}, T_{l+1}X-T_l, g_{l+1}(T_{l+1}))= A_{l+1} \, .\] 
Note that the components $L_{l+1,i}$ of $\lblow{l+1}=\Spec A_{l+1}$ are the components $L_{l+1,i}$ of $\Spec \widetilde{A}_{l+1}$.
  \end{prf}
  
 \begin{Prf}{Theorem \ref{thm:F_Nreg}}
 According to Lemma \ref{lem:part1} the only singular points are closed points on the
 components of type $A$ and type $B$. Let $L_{(x_a,y_a)}$ be a component of type $A$ that
 corresponds to a prime ideal $\frakP=(\pi, X, Y-\zeta_m^i)\subset S_1^{sh}$. We 
 work in the affine open subset $\blow{}=\Spec A$, where $A$ is the ring of \eqref{eq:14}.
 We blow up $\blow{}$ along $V(\frakP A)$. Since $\frakP A=(\pi, X)$, the blow-up is covered by
 two affine open subsets. Setting $T_1=\frac{\pi}{X}$, the first one is given by $\lblow{1}$. The
 only new components are $L_{1,1}, \ldots , L_{1,p}$, cf. Figure \ref{fig:A_l} with $l=1$.
 Setting $X_1=\frac{X}{\pi}$, the second subset is \[ \Spec S[X_1]/ (X_1 \pi - X, g(X_1)) \, ,\] where \[ 
g(X_1)= \bW{1}^p+\mu \frac{\psi((X_1 \pi)^m,1-(X_1 \pi)^m)}{\pi}+\mu\sum_{k=1}^{p-1} {p \choose
k}p^{-1}{\bW{1}}^{p-k}\pi^{p-k-1}Y^{mk}(-1)^k \, . \] Here we only have to study the prime
ideals $\frakm$ such that $X_1, \pi \in \frakm$, since all the others that lie above $\pi$ can
be found in $\lblow{1}$. We have \[ \bW{1}^p+\mu \bW{1} = \pi P(X_1)\] in $S[X_1]/(X_1 \pi - X, g(X_1))$,
where $P(X_1)\in S[X_1]$. It follows that $\bW{1}^p+\mu \bW{1} \in \frakm$, which implies
\begin{equation}\label{eq:Aroots} \bW{1}\in \frakm \mbox{ or } \bW{1}-\theta\zeta_{p-1}^i \in \frakm
\end{equation} for some $0 \leq i \leq p-2$; here $\theta\in R^{sh}$ satisfies
$\theta^{p-1}=-\mu$. The prime ideal $\frakm$ is of the form $ \frakm=(\pi,X_1, \bW{1})$
($\frakm=(\pi, X_1, \bW{1}-\theta\zeta_{p-1}^i )$ resp.), hence maximal. In fact, they are the
``end points'' of the components $L_{1,i}$. Since the factors in \eqref{eq:Aroots} are
pairwise coprime, \[ \frakm \left(S[X_1]/ (X_1 \pi -X, g(X_1))\right)_\frakm \] is generated by two elements, hence $S[X_1]/ (X_1 \pi -X, g(X_1))$ is regular at $\frakm$. 
There are $p$ singular closed points lying on $L_{(x_a,y_a)}$ (Lemma \ref{lem:part2}). If we
blow up this line, we get further components $L_{2,1}, \ldots, L_{2,p}$ by Lemma~\ref{lem:part3}. 
There are no singular closed points lying on the $L_{1,i}$, see Lemma~\ref{lem:part3}.
 Lemma~\ref{lem:part2} implies that the only singular closed points that lie on the $L_{2,i}$ or the line
$L_{(x_a,y_a)}$ are the points where the $L_{2,i}$ intersect $L_{(x_a,y_a)}$. It is clear that repeating this process (i.e. blowing up the component
$L_{(x_a,y_a)}$) $m-3$ times gives the resolution of the singularities that lie on
this component, and therefore yields the configuration we claimed.  
By symmetry we can argue analogously for components of type $A$ which correspond to prime ideals of the form $\frakP=(\pi, X-\zeta_m^i, Y)$.

Finally, a similar (but simpler, since no inductive argument is needed) computation shows
that we have to blow up the components of type $B$ only once, yielding the remaining
assertions of the lemma.
\end{Prf}

\subsection{The configuration of the geometric special fiber of the local minimal regular model}
  
Having located and resolved the singularities of $\calX^{sh}$, we can now describe the
minimal regular model of $F_N$ over $R$.

\begin{thm}\label{thm:F_Nmin}
Let $N$ be an odd squarefree natural number which has at least two prime factors, $\zeta_N$
a primitive $N$-th root of unity and $N=pm$, where $p$ is prime and $m \in \NN$.
Furthermore, let $R$ be the localization of $\ZZ[\zeta_N]$ with respect to a fixed prime
ideal $\frakp \in \Spec \ZZ[\zeta_N]$ that lies above $p$. We denote by
$\mathfrak{F}^{min}_{N,\frakp}\rightarrow \Spec R$ the minimal regular model of the Fermat
curve $F_N$ over $R$. Then the geometric special fiber  \[\mathfrak{F}_{\pi} \colonequals 
\mathfrak{F}^{min}_{N,\frakp}\times_{\Spec R} \Spec \overline{\FF}_p \] has the
configuration as in Figure \ref{fig:minmodel}; Table \ref{tabular:quant} contains the
number, multiplicity, genus and self-intersection of the components. Finally, all
intersection between components of the geometric special fiber are transversal. 
\end{thm}

\begin{figure}[h!]\centering
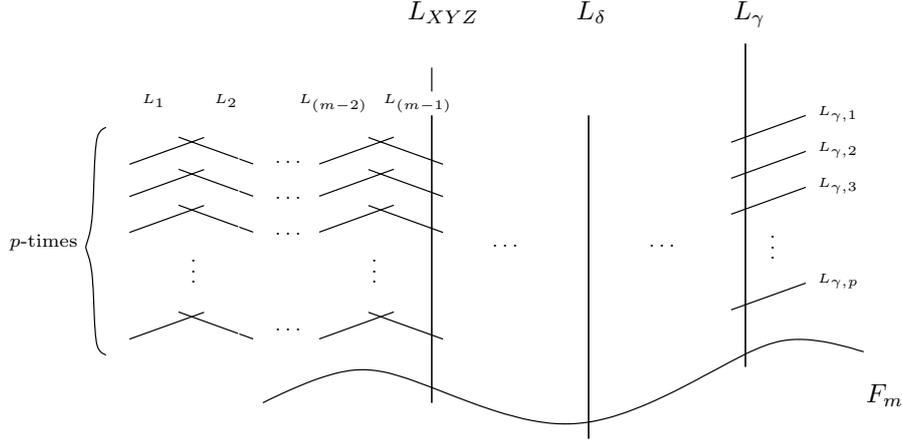
\parbox{12cm}{\caption{The configuration of the geometric special fiber  $ \mathfrak{F}_{\pi} $.}\label{fig:minmodel}}
\end{figure}

\begin{table}[h!]\centering
\begin{tabular}{c | c | c | c | c}
                                    & Number of components & Multiplicity & Genus & Self-intersection \\  \hline
   $L_i$                        & $3mp$                               & $i$   & $ 0$ & $-2$ \\
   $L_{XYZ}$               &            $3m$                      & $m$            & $0$  & $-p$ \\
   $L_{\gamma}$     &           $m\varrho$                                 &  $2$            &  $0$ & $-p$ \\
   $L_{\gamma,j}$ &            $pm\varrho$                                &  $1$           &  $0$ & $-2$\\
   $L_\delta$                 &            $m^2(p-3)- 2m\varrho $                          & $1$             & $0$ & $-p$\\
  $F_m$                         &        $1$                            & $p$             & $\frac{1}{2}(m-1)(m-2)$ & $-m^2$  \\
\end{tabular}
\caption{$\varrho$ denotes the number of factors with multiplicity two of $\Psi (X^m)$ over $\overline{\FF}_p$ (cf. Definition \ref{def:scrW}).\label{tabular:quant}} 
\end{table}

\begin{prf}
The scheme \[ \mathfrak{F}^{0}_{N,\frakp}= \Proj R[X_0,Y_0, Z_0]/ (X_0^N+Y_0^N-Z_0^N) \]
is covered by the affine scheme $\calX$ in \eqref{sch:F_N} and by \[ \calX^{\prime}= \Spec
R[Y^{\prime},Z^{\prime}]/(1+{Y^{\prime}}^N-{Z^{\prime}}^N )\, , \] where
$Y^{\prime}=\frac{Y_0}{X_0}$ and $Z^{\prime}=\frac{Z_0}{X_0}$. To blow up
$\mathfrak{F}_{N,\frakp}^{0}$ along the ideal $V_{+}(X_0^m+Y_0^m-Z_0^m, \pi)$ is to blow
up $\calX$ along $(\pi, F_m)$ and $\calX^{\prime}$ along $(\pi,
1+{Y^{\prime}}^m-{Z^{\prime}}^m)$ and then glue the resulting schemes together; we denote these
blow-ups by $\widetilde{\calX}$ and $\widetilde{\calX}^{\prime}$. As $\calX$ is
isomorphic to $\calX^{\prime}$ and $(\pi, F_m)$ to $(\pi,
1+{Y^{\prime}}^m-{Z^{\prime}}^m)$ via $X\mapsto Z^{\prime}$ and $Y\mapsto -Y^{\prime}$, the
blow-ups $\widetilde{\calX}$ and $\widetilde{\calX}^{\prime}$ are isomorphic as well. The
only components of $\widetilde{\calX}^{\prime}$ which are not in $\widetilde{\calX}$ are
the ones corresponding to prime ideals that contain $Z^{\prime}$. Under the
isomorphism above these components correspond to the components of type $A$ which contain
$X$. It follows that we can apply Theorem \ref{thm:F_Nreg} to resolve the singularities of these schemes.
The regular model of $F_N$ we obtain in this way will be denoted by
$\mathfrak{F}_{N,\frakp}$. By the discussion above, it is enough to
analyze the regular scheme from Theorem \ref{thm:F_Nreg}, remembering that there are a
few more components which we cannot see in this affine open subset. 
We sketch how the quantities in Table~\ref{tabular:quant} can be derived. In fact, we 
compute these quantities for the model $\mathfrak{F}_{N,\frakp}$, we will see later that
in fact $\mathfrak{F}_{N,\frakp}=\mathfrak{F}_{N,\frakp}^{min}$. 

Let us start with the number of
components of $\mathfrak{F}_{N,\frakp}$. By Theorem \ref{thm:F_Nreg} it is clear that
the geometric special fiber of $\mathfrak{F}_{N,\frakp}$ has the configuration depicted in
Figure~\ref{fig:minmodel}. The vertical components are parametrized by pairs $(x,y)\in
\overline{\FF}_p$ with $x^m+y^m-1=x^m\prod_{i=0}^{m-1} (x-\overline{\zeta}_m^i)
\overline{\Psi}(x^m) = 0 $, see Proposition \ref{prop:normalF_N}. There are $\varrho$ factors
$(X-\overline{\gamma}_k)^2$ in $\overline{\Psi}(X^m)$, and for each $\overline{\gamma}_k$
the polynomial $Y^m+\overline{\gamma}_k^{m}-1 \in \overline{\FF}_p [Y]$ has $m$ solutions,
as $\overline{\gamma}_k^m \neq 1$. Hence we get $m\varrho$ lines. We denote
these lines by $L_{\gamma}$; they are the ones of type $B$ in Theorem \ref{thm:F_Nreg}. Furthermore, there are $m(p-3)-2\varrho$ linear factors $(X-\overline{\delta})$ and with the same argument as before there are $m(m(p-3)-2\varrho)$ lines which correspond to these. We denote these by $L_{\delta}$. 

The only solutions which are left are the following:
\begin{equation}\label{linesX}
(0,\overline{\zeta}_m^i)
\end{equation} for $0\leq i\leq m-1$, and 
\begin{equation}\label{linesY}
(\overline{\zeta}_m^i,0)
\end{equation} for $0\leq i \leq m-1$.
This gives us $2m$ lines; they are the components of type $A$ in Theorem \ref{thm:F_Nreg}.
However, as mentioned above, there are more lines which behave like the ones of type $A$ but
which cannot be seen in this affine picture. In fact, by the isomorphism we described at
the beginning of the proof, it is clear that there are $m$ more lines, hence these, together with the
ones of \eqref{linesX} and \eqref{linesY}, give us $3m$ lines. We denote them by $L_{XYZ}$.
According to Theorem \ref{thm:F_Nreg}, for each $L_{XYZ}$ there are $p$ chains of $m-1$
lines, where the ends of the chains intersect $L_{XYZ}$. These ends are denoted by
$L_{(m-1)}$ and the following lines by $L_{(m-2)}, L_{(m-3)},$ etc. Also by
Theorem~\ref{thm:F_Nreg}, there are $p$ lines intersecting each $L_{\gamma}$. We denote these lines by $L_{\gamma,1}, \ldots, L_{\gamma,p}$. Collecting this information we get the number of components of table \ref{tabular:quant}. \\

Next, we want to study the multiplicity of the components in the geometric special fiber
$\mathfrak{F}_\pi$, see~\cite[Definition 7.5.6]{Liu}.
We illustrate this only in a few cases. For example, let us return to the scheme $\lblow{l}=\Spec A_l$ in \eqref{eq:A_l}. The
prime ideals of height 1 of $A_l$ are $(\pi, X, \bW{1})$ and $(\pi, X, \bW{1}-\theta \zeta_{p-1}^i)$
for $0 \leq i \leq p-2$. These correspond to the components $L_l$. Furthermore, there is
the prime ideal $(\pi, X, T_l)$ which corresponds to a component $L_{XYZ}$, after blowing up $m-1-l$
times. Let $\frakP$ be a prime ideal that corresponds to $L_l$. In Theorem
\ref{thm:F_Nreg} we have seen that $\frakP (A_l)_\frakP=(X)$. Since $\pi=T_l X^l$ in $A_l$
and $T_l$ becomes a unit in $(A_l)_\frakP$, we get $\nu_{L_l} (\pi)=l$, hence the
multiplicity of $L_l$ is $l$. Now let $\frakP=(\pi,X,T_l)$. Equation \eqref{eq:T_l} shows
$T_l=X^{m-l} \epsilon$ in $(A_l)_\frakP$, where $\epsilon \in
(A_l)_\frakP^{\ast}$. With the same argument as before we get $\nu_{L_{XYZ}} (\pi)=m$,
hence the component $L_{XYZ}$ has multiplicity $m$. The multiplicities of the other
components can be computed in a similar way.
The genera of the components are obvious. 

We now prove that all intersections are transversal. Let $\calT$ denote the set of irreducible components
of $\mathfrak{F}_\pi$. Then we have
\[
  \mathfrak{F}_\pi = \sum_{\calC \in \calT}d_\calC\,\calC\, ,
\]
where $d_\calC$ is the multiplicity of $\calC$ in $\mathfrak{F}_\pi$.
For a component $\calC \in \calT$, we have \[ 0< \calC (\mathfrak{F}_\pi - d_\calC \calC)
\, .\] Let us
denote by $I_\calC$ the sum of the multiplicities of the components that have a positive
intersection number with $\calC$. Obviously we have \[ I_\calC \leq \calC
(\mathfrak{F}_\pi - d_\calC \calC)\, , \] and equality holds for all $\calC$ if and only if
all intersections are transversal. We get the following table:
\begin{center}
\begin{tabular}{c | c}
$\calC$ & $I_\calC$\\ \hline
$L_i$ & $2i$ \\
$L_{XYZ}$ &$ p + p(m-1)$ \\
$L_{\gamma} $& $2p $\\
$L_{\gamma,j}$ & $2$ \\
$L_\delta $& $p $\\
$F_m$ &$ m^2 p$ 
\end{tabular}
\end{center}

Let us denote by $\calK$ a canonical divisor of $\mathfrak{F}_{N,\frakp}$. By the
adjunction formula (cf. Theorem~\ref{thm:kan}) and by properties of the intersection matrix
of $\mathfrak{F}_{\pi}$ (see for instance~\cite[Proposition~8.1.21,
Proposition~8.1.35]{Liu}) we have 
\jot3mm
\begin{align*}
2 g_a (F_N)-2 &= \calK \cdot \mathfrak{F}_\pi \\
              &= \sum_{\calC \in \calT} d_\calC (\calK \cdot \calC) \\
              &= \sum_{\calC \in \calT} d_\calC (-\calC^2+ 2g_a (\calC) -2) \\
              &= \sum_{\calC \in \calT} \calC (\mathfrak{F}_\pi- d_\calC \calC )+2p g_a
  (F_m)- 2 \sum_{\calC \in \calT} d_\calC \\
  & \geq  \sum_{\calC \in \calT} I_\calC + 2p g_a (F_m) -2\sum_{\calC \in \calT} d_\calC \, ;
\end{align*}
\jot1mm
 hence the intersections are transversal if and only if \begin{equation}\label{eq:trans} 
   2g_a (F_N)-2= \sum_{\calC \in \calT} I_\calC + 2p g_a (F_m) - 2\sum_{\calC \in \calT} d_\calC \, .
\end{equation}
Using the  quantities of Table \ref{tabular:quant} and the table for the $I_\calC$ we get \[
\sum_{\calC \in \calT} I_\calC= 3m^3p-2m^2p+2pm\varrho+m^2p^2  \] and \[
-2\sum_{\calC \in \calT} d_\calC= -3m^3p+m^2p-2pm\varrho-2p \, . \]  We have \[
2 g_a (F_N)-2= m^2p^2 -3mp \]
and 
\jot3mm
\begin{align*}  \sum_{\calC \in \calT} I_\calC  -2\sum_{\calC \in \calT} d_\calC+ 2p g_a (F_m)&= -m^2p+m^2p^2-2p+p(m-1)(m-2) \\
&= m^2p^2-3mp \, , 
\end{align*}
\jot1mm
which yields~\eqref{eq:trans} and therefore the transversality of the intersections. \\
Since we know the intersection numbers and the configuration of the geometric special
fiber, one can use that $(\calC\cdot \mathfrak{F}_\pi) = 0$ to get the self-intersection
number of a component $\calC \in \calT$. \\
Finally, since there are no exceptional divisors by Castelnuovo's criterion~\cite[Theorem~9.3.8]{Liu},  $\mathfrak{F}_{N,\frakp}$ is already
the minimal regular model.
\end{prf}

\begin{rem}\label{sqfpr}
  If we consider the case $m=1$, so that $N=p$ is prime, then the model constructed in
  Theorem~\ref{thm:F_Nmin} remains regular. However, the component $F_m = F_1$ is an
  exceptional divisor, so the model is not minimal. Contracting $F_1$ yields the minimal
  regular model of $F_p$ over $R$, see~\cite{Mc}.
\end{rem}

We can use Theorem~\ref{thm:F_Nmin} to analyze the singularities of the normalization 
$\mathfrak{F}_{N,\frakp}^{nor}$ of the scheme \[
\mathfrak{F}_{N,\frakp}^0=\Proj R[X,Y,Z]/(X^N+Y^N-Z^N) \, . \] 
Recall that a normal and excellent two-dimensional scheme $\calX$ has \emph{rational
singularities}, if for one (and hence every) desingularization
$f:\calX^{\prime}\to \calX$, we have 
\[
  R^i f_\ast \calO_{\calX^{\prime}} = 0
\]
for all $i>0$. See~\cite{Artin}.

\begin{coro}
The normal scheme $\mathfrak{F}_{N,\frakp}^{nor}$ has rational singularities. 
\end{coro}
\begin{prf}
It follows from the proofs of Theorems~\ref{thm:F_Nreg} and~\ref{thm:F_Nmin} that there
is a desingularization $f^{nor}:\mathfrak{F}_{N,\frakp}^{min} \rightarrow \mathfrak{F}_{N,\frakp}^{nor}$.
Let $P\in \mathfrak{F}_{N,\frakp}^{nor}$ be a singular point and $\calC_1, \ldots,
\calC_n$ the components of $\mathfrak{F}_{N,\frakp}^{min}$ with $f^{nor} (\calC_i)=P$.
According to~\cite[Theorem~3]{Artin}, $P$ is a rational singularity if and only if the
fundamental cycle $\calZ_P$ with respect to $P$, also defined in~\cite{Artin},
satisfies $p_a(\calZ_P)=0$ . Using Theorem \ref{thm:F_Nmin}, we find that
 \[ \calZ_P= \sum_{i=1}^{n} \calC_i \, .\] The adjunction formula together with an inductive argument yields \[
p_a (\calZ_P)=\sum_{i=1}^n p_a (\calC_i) +\sum_{1 \leq i < j \leq n } (\calC_i \cdot
\calC_j) -(n-1)=\sum_{1 \leq i < j \leq n } (\calC_i \cdot \calC_j) -(n-1)\, .\] Finally,
it is not hard to see -- using the configuration described in Theorem \ref{thm:F_Nmin} -- that $p_a (\calZ_P)=0$.
\end{prf}

\begin{rem}\label{rem:other}
  The computation of local minimal regular models of Fermat curves of squarefree even or
  squareful exponent is more involved. See~\cite[Chapter~7]{CurillaThesis} for a summary
  of the problems one encounters and possible strategies for dealing with them.
\end{rem}

\section{The global minimal regular model}\label{sec:global}
Let $N$ be an odd squarefree composite integer.
In this section we turn to the global situation; we construct the minimal regular model of $F_N$ over
$\ZZ[\zeta_N]$, where $\zeta_N$ is a primitive $N$-th root of unity.
The following result shows that it essentially suffices to localize at the primes $\frakp$
of $\ZZ[\zeta_N]$ dividing $N$.

\begin{prop}\label{prop:goodprimes}
Let $\calX$ be the Fermat scheme
\[\calX=\Spec \ZZ[\zeta_N] [X,Y]/(X^N+Y^N-1) \, .\]
If $\frakp$ is a prime ideal of $\ZZ[\zeta_N]$ not dividing $N$, then $\calX$ is regular
at $\frakp$.
\end{prop}
\begin{prf}
We have a morphism $g:\calX \rightarrow \calY=\Spec \ZZ[\zeta_N]$ which corresponds to the
ring homomorphism \[g^{\sharp}:\ZZ[\zeta_N] \rightarrow \ZZ[\zeta_N] [X,Y]/(X^N+Y^N-1)\]
where $g^{\sharp}$ is the composition of the inclusion $\ZZ[\zeta_N] \rightarrow
\ZZ[\zeta_N][X,Y]$ and the canonical surjection $\ZZ[\zeta_N][X,Y]\rightarrow\ZZ[\zeta_N]
[X,Y]/(X^N+Y^N-1)$. The scheme $\calX$ is integral, $\calY$ is a Dedekind scheme, and $g$
is non-constant, hence the morphism $g$ is flat, see e.g. \cite{Liu}, p.137: Corollary
3.10.). We want to show that $\calX$ is regular at a prime ideal $\frakp \in \calX$ if
$N\not\in \frakp$. To see this we start with a prime ideal $\frakp$ with $g(\frakp)=0$.
Then this prime ideal is the image of an element of $\calX_{\QQ(\zeta_N)}=\Spec
{\QQ(\zeta_N)}[X,Y]/(X^N+Y^N-1)$ with respect to the obvious morphism
$\calX_{\QQ(\zeta_N)} \rightarrow \calX$. Since this morphism is flat and
$\calX_{\QQ(\zeta_N)}$ is regular it follows that $\calX$ is regular at $\frakp$ (see e.g.
\cite{ega4_2}, p.143: Corollaire 6.5.2.). Next, let $\frakp$ be a prime ideal with
$g(\frakp)=\frakq$, where $\frakq$ is a prime in $\ZZ[\zeta_N]$. Since $\calY$ is regular,
we only have to concentrate on the fiber $\calX_{\frakq}=\Spec k(\frakq)
[X,Y]/(X^N+Y^N-1)$, where $k(\frakq)$ is the residue field of $\frakq$ (Lemma~\ref{lemma:Reg2}). We use the Jacobian criterion to analyze the scheme $\calX_{\frakq}$.
For simplicity we may change to the \emph{geometric special fiber}\index{special
fiber!geometric} $\overline{\calX}_{\frakq}=\calX_{\frakq} \times_{\Spec k(\frakq)} \Spec
\overline{k(\frakq)}=\Spec \overline{k(\frakq)} [X,Y]/(X^N+Y^N-1)$. Since the inclusion
morphism $k(\frakq) \hookrightarrow \overline{k(\frakq)}$ is faithfully flat, the
projection morphism $p_2:\overline{\calX}_{\frakq}\rightarrow \calX_{\frakq}$ is
faithfully flat as well. Hence, if $\overline{\calX}_{\frakq}$ is regular, then
$\calX_{\frakq}$ is regular, see Remark~\ref{rk:ffreg}. Now let us assume that $N\notin \frakq$. Then the rank of
the Jacobian matrix $J=(NX^{N-1},NY^{N-1})$ is 1 for all points of
$\overline{\calX}_{\frakq}$ and so $\overline{\calX}_{\frakq}$ is regular by the Jacobian
criterion and by \cite[Corollary~4.2.17.]{Liu}, hence $\calX$ is regular in
$\frakp$ (Lemma~\ref{lemma:Reg2}). If $N\in \frakq$ then the Jacobian matrix is zero
and it follows that $\overline{\calX}_{\frakq}$ is singular at all points. In this
situation Lemma~\ref{lemma:Reg2} does not tell us, if $\calX$ is regular at $\frakp$.
\end{prf}

We now use Theorem~\ref{thm:F_Nmin} and Proposition~\ref{prop:goodprimes} to determine the minimal regular model of $F_N$ over
$\ZZ[\zeta_N]$.
Let $U = \Spec \ZZ[\zeta_N,1/N]\subset \Spec \ZZ[\zeta_N]$ be the open subset consisting of the prime ideals
$\frakp$ with $N\notin \frakp$. We set
$\mathfrak{F}_{N,U}^{min}=\mathfrak{F}_N^0 \times_{\Spec \ZZ[\zeta_N]} U$, where \[
\mathfrak{F}_N^0=\Proj \ZZ[\zeta_N] [X,Y,Z]/(X^N+Y^N-Z^N)\, ;\] the scheme
$\mathfrak{F}_{N,U}^{min}$ is regular by Proposition \ref{prop:goodprimes}. For a prime
ideal $\frakp$ with $N\in \frakp$, recall the minimal regular model
$\mathfrak{F}_{N,\frakp}^{min}$ from Theorem \ref{thm:F_Nmin}, where $\frakp \cap
\ZZ=(p)$.

\begin{coro}\label{coro:F_Nmin}
The minimal regular model
$\mathfrak{F}_{N}^{min}$ of the Fermat curve $F_N$ over $\Spec \ZZ[\zeta_N]$
can be obtained by gluing the scheme $\mathfrak{F}_{N,U}^{min}$ and all the
$\mathfrak{F}_{N,\frakp}^{min}$, where $\frakp$ runs through all primes of $\ZZ[\zeta_N]$
dividing $N$.
\end{coro}

\begin{proof}
It follows from general descent theory (cf.~\cite[Chapter~6]{BLR}) that we can glue 
$\mathfrak{F}_{N,U}^{min}$ and the $\mathfrak{F}_{N,\frakp}^{min}$ to get a regular model
of $F_N$ over $\Spec(\ZZ[\zeta_N])$. See~\cite[Corollary~2.3.5]{CurillaThesis} for a precise
statement.
This model is indeed the minimal regular model, since it contains no exceptional divisors by Castelnuovo's criterion~\cite[Theorem~9.3.8]{Liu}. 
\end{proof}

\newpage

\section*{\large Part~II: The arithmetic self-intersection of the relative dualizing sheaf
  on the minimal model of a Fermat curve of odd squarefree exponent}
  \section{Bounding the arithmetic self-intersection of the relative dualizing sheaf on
  arithmetic surfaces}\label{sec:bound_om}
\subsection{Arakelov intersection theory on arithmetic surfaces}\label{sec:arakelov}
Throughout this section we let $K$ be a number field, $\calO_K$ its ring of integers and $\pi:\calX\rightarrow \Spec
\calO_K$ an arithmetic surface whose generic fiber $X$ has genus $\ge 2$. 
See Soul\'e~\cite{Soul} and~\cite{CK} for the definitions and results on 
intersection multiplicities between hermitian line bundles that we need in the following.
In fact, we will only encounter intersection multiplicities between certain special
hermitian line bundles.
On the one hand, we consider hermitian line bundles $\overline{\calO(V)}$, where $V = \sum_{\frakp}
V_\frakp$ is a vertical divisor on $\calX$ with the sum running over all closed points
$\frakp \in \Spec \calO_K$, and the metric is trivial. For instance, we then
have
\begin{equation}\label{vertintsum}
  \overline{\calO(V)}^2 = \sum_{\frakp}V_\frakp^2\log \Nm(\frakp)\, .
\end{equation}
On the other hand, we consider the hermitian line bundle $\overline{\omega}_\calX=
(\omega_\calX, \|\cdot\|)$,
where $\omega_\calX = \omega_{\calX/\calO_K}$
is the relative dualizing sheaf of $\calX$ over $\calO_K$
and $\| \cdot \|$ is the \emph{Arakelov metric}, i.e. the unique metric on $\omega_\calX$ such
that the Arakelov adjunction formula holds, see~\cite[\S4]{Ar}.
The goal of Part~II is to bound $\overline{\omega}_\calX^2$ in terms of $N$ when $\calX$ is
the minimal regular model of a Fermat curve of odd squarefree exponent $N$ over
$\ZZ[\zeta_N]$.
\begin{rem}\label{omega_comp}
    Instead of $\omega_\calX = \omega_{\calX/\calO_K}$, some authors prefer to work with the relative dualizing sheaf
$\omega_{\calX/\ZZ}$, also equipped with the Arakelov metric.
We have
\[
  \omega_{\calX/\ZZ} = \omega_{\calX/\calO_K}\otimes
  \pi^\ast\omega_{\calO_K/\ZZ}\,,
\]
Therefore
\begin{equation}\label{om_def}
  \overline{\omega}_{\calX/\ZZ}^2 = \overline{\omega}_{\calX/\calO_K}^2 +
  (2g-2)\log|\Delta_{K|\QQ}|^2\,,
\end{equation}
so that bounds on $\overline{\omega}^2_{\calX/\calO_K}$ are easily translated into bounds on
$\overline{\omega}^2_{\calX/\ZZ}$ and vice versa.
\end{rem}

\subsection{K\"uhn's upper bound}\label{sec:kuehn}
We first recall a method for the computation of an upper bound on $\omega_\calX$ due
to K\"uhn~\cite{omega}.
Let $\calY \to \Spec \calO_K$ be an arithmetic surface with generic fiber $Y$. 
Fix $\infty , P_1 , ...,P_r \in Y(K)$ such that $Y \setminus
\{\infty,P_1, ...,P_r\}$ is hyperbolic. In this section we assume that the arithmetic surface $\calX\to \Spec \calO_K$ 
comes equipped with a dominant morphism 
$\belyi: \calX \to \calY$ of degree $d$ such that the induced  morphism
$\belyi: X \to Y$  
is  unramified outside $\infty ,P_1,...,P_r$.
We write $\belyi^* \infty = \sum b_j S_j$ and set $b_{\max} = \max_j \{b_j\}$. 
We call a prime $\frakp$ \emph{bad}\index{bad prime} if the fiber $\calX_\frakp$ of
$\calX$ above $\frakp$ is reducible, in which case $\calX_\frakp$ is called a \emph{bad
fiber}.
K\"uhn has shown how to bound $\omega_\calX^2$ in terms of data which depends only on $K$,
on $Y$, on $b_{\max}$ and on the configuration of the bad fibers of $\calX$.

Let $\calK$ be a canonical $\QQ$-divisor of $\calX$.
For each $S_j$ we can find a $\QQ$-divisor
$\calF_j$ such that
\begin{align}\label{eq:F_i}
 \left( \calS_j + \calF_j -\frac{1}{2g-2} \calK\right )\cdot\mathcal{C}=0
\end{align}
for all vertical irreducible components $\mathcal{C}$ of $\calX$. Similarly we can find, 
for each $S_j$, a $\QQ$-divisor $\calG_j$ such that for all vertical irreducible
 components $\mathcal{C}$  we have
\begin{align}\label{eq:G_i}
\left ( \calS_j + \calG_j -\frac{1}{d} \div(s) \right) \cdot
\mathcal{C}=0  \, ,
\end{align} 
where $\overline{\infty}$ is the Zariski closure of $\infty$ in $\calY$ and $s$ is a
section of $\belyi^*{\calO(\infty)}$.
We define
\begin{align}\label{eq:def-ap}
 \sum_{\frakp \,\,{\rm bad}} a_\frakp
  \log\Nm(\frakp)
&= - \frac{2 g 
}{d} \sum_j b_j\,\overline{\mathcal{O}}( \calG_j)^2 +
\frac{2g-2}{d} \sum_j b_j\, \overline{\calO}(\calF_j)^2\, ,
\end{align} 
where the line bundles carry the trivial metric.

\begin{thm} \label{thm:keyformula} Let $\belyi: \calX \to \calY$ be as above.
If all $S_j$ are $K$-rational points and all divisors of degree zero supported in the
$S_j$ are torsion, then
  the arithmetic self-intersection number of the dualizing sheaf $\overline{\omega}_{\calX}$ on  $\calX$ 
  satisfies the inequality
\begin{align*}\label{eq:main}
  \overline{\omega}_{\calX}^2 &\le (2 g-2) \left( 
   [K:\QQ]  \left(  \kappa_1 \log b_{\max} + \kappa_2 \right) + \sum_{\frakp \,\,{\rm bad}} a_\frakp
  \log\Nm(\frakp)
   \right),
\end{align*}
where $\kappa_1, \kappa_2$ are positive real constants
that depend only on $Y$ and the points $\infty,P_1,...,P_r$.
\end{thm}

\begin{prf}
  This follows from \cite[Theorem~I]{omega} and~\eqref{om_def}.
\end{prf}

The real number $ \sum_{\frakp \,\,{\rm bad}} a_\frakp \log\Nm(\frakp)$ is called the
\emph{geometric contribution}.
  Upper bounds for the geometric contribution which are easily computed from the
  configuration of the special fibers of $\calX$ can be found in~\cite[\S6]{omega}.
  The real number $ [K:\QQ]  \left(  \kappa_1 \log b_{\max} + \kappa_2
  \right)$ is called the \emph{analytic contribution}\index{analytic contribution}.
  
\subsection{Lower bounds}\label{sec:km}
Let $S \in X(K)$ be a rational point with Zariski closure $\calS \in \Div(\calX)$ and
let $V_S \in \Div_\QQ(\calX)$ denote a vertical $\QQ$-divisor such that 
\begin{equation}\label{eq:vsglob}
    (\calS + V_S) \cdot \calC = \frac{a_\calC}{2g-2}
\end{equation}
holds for all vertical irreducible components $\calC$ of $\calX$,
where $a_\calC$ is defined in \eqref{a_c}.
Such a $\QQ$-divisor exists by~\cite[Proposition~2.1]{KuehnMueller}.
According to~\cite[Corollary~2.3]{KuehnMueller}, we can also find, 
for every vertical irreducible component $\calD$ of $\calX$, 
a vertical $\QQ$-divisor $V_{\calD} \in \Div_\QQ(\calX)$ such that 
\[
  (V_{\calD} \cdot \calC) = \frac{a_{\calC}}{2g-2} - \frac{\delta_{\calD,\calC}}{d_\calD}\, ,
\]
holds for all vertical irreducible components $\calC$ of $\calX$,
where $d_\calD$ is the multiplicity of $\calD$ in the special fiber of $\calX$ containing
it and $\delta$ is the Kronecker delta on the set of irreducible components.
We set 
\[
  U_S = \sum_\calC d_\calC(2(V_\calC \cdot V_S) -V_\calC^2)\,\calC
\]
and 
\[
  \beta_S = \frac{1-g}{g}\;\overline{\calO}(2V_S+U_S)^2+2(\bar{\omega}_\calX\cdot
  \overline{\calO}(U_S))\, ,
\]
where the vertical line bundles are equipped with the trivial metric.
In~\cite{KuehnMueller}, K\"uhn and the second author used this to find a method for
computing a lower bound for $\om_{\calX}$.

\begin{thm}\label{thm:km}
  With notation as above, suppose that 
\begin{enumerate}[(i)] 
  \item $(2g-2)S$ is a canonical divisor on $X$;
  \item we have 
      \begin{equation}\label{eq:rel_semipos}
      a_\calC+2(\calS\cdot \calC) - (U_S\cdot \calC) \ge 0
      \end{equation}
    for all vertical irreducible components $\calC$ of $\calX$.
\end{enumerate}
Then we have
\[
  \om^2_{\calX} \ge \beta_S\, .
\]
\end{thm}
\begin{proof}
  See Proposition~1.2 and Theorem~1.3 of~\cite{KuehnMueller}.
\end{proof}
One can show that in favorable situations (for instance, when $\calX$ has only reduced
special fibers and at least one of its special fibers is reducible), condition (i) 
can be dropped and condition (ii) is always satisfied
and that $\beta_S$ is a positive lower bound for $\om^2_{\calX}$. 
However, For our intended application to $\calX =\fnm$, we will have to check conditions (i) and
(ii) and the positivity of $\beta_S$.

\section{Computations on the local minimal regular model}\label{sec:local_comps}
Let $N$ be an odd squarefree natural number which has at least two prime factors, let $\zeta_N$
be a primitive $N$-th root of unity and let $F_N/\QQ(\zeta_N)$ denote the Fermat
curve~\eqref{f_n}.  The minimal
regular model $\fnm$ of $F_N$ over $\Spec \ZZ_N$ was constructed in Part~I.
In order to bound $ \omega_{\fnm}^2$ using Theorems~\ref{thm:keyformula}  and~\ref{thm:km} we need to show that these results are indeed applicable and we need
to compute the
quantities appearing in their statements.
We recall the following notation from Section~\ref{sec:curilla}:
Let $N=pm$, where $p$ is prime and $m \in \NN$.
Fix a prime $\frakp$ of $\ZZ[\zeta_N]$ above $p$ and let 
$R$ be the localization of $\ZZ[\zeta_N]$ with respect to $\frakp$.
The minimal regular model $\mathfrak{F}^{min}_{N,\frakp}\rightarrow \Spec R$ of the Fermat
curve $F_N$ over $R$ is described explicitly in Theorem~\ref{thm:F_Nmin}.  
We will mostly work on the base change
$\fnpm\times_{\Spec R} \Spec R^{sh} $, where $R^{sh}$ is the strict Henselization of $R$. We denote 
the special fiber of this model by $
\mathfrak{F}_\pi = \mathfrak{F}^{min}_{N,\frakp}\times_{\Spec R} \Spec \overline{\FF}_p$ .

\subsection{Local extensions of cusps}\label{sec:cusps}
Consider the Galois covering 
\begin{align}\label{derMor}
\belyi:F_N\rightarrow \mathbb{P}^1
\end{align}
of degree $N^2$ given by $(x:y:z)\mapsto (x^N:y^N)$.  
In fact $\belyi$ is a Belyi morphism, because it is 
a unramified outside $0,1,\infty$, and is defined over $\QQ$ with 
ramification orders all equal to $N$; see~\cite{MurtyRama} for a discussion of the associated Belyi uniformization. 
In \S\ref{sec:upper_fermat}, we will use $\belyi$ to compute an upper bound on
$\omm^2$ using Theorem~\ref{thm:keyformula}.
We call the ramification points of $\belyi$ the {\em cusps} of $F_N$. 
A divisor on $X$ is called \emph{cuspidal}\index{divisor!cuspidal} if all points in its
support are cuspidal. 
We now investigate the Zariski closures of the cusps inside the
minimal regular model.

\begin{nota}\label{nota:cusp}
Assume that we have fixed a primitive $N$-th root of unity $\zeta_N$. Then we denote by $S_{x_i}$
($S_{y_i}$, $S_{z_i}$, resp.) the cusp $(0:\zeta_N^{i}:1)$ ($(\zeta_N^{i}:0:1)$,
$(\zeta_N^{i}:-1:0)$, resp.). If the properties of the cusp, which are relevant for our
consideration, do not depend on the exponent $i$ we drop the subscript and just write
$S_x$ ($S_y$, $S_z$, resp.). For a normal model of the Fermat curve the Zariski closure of
a cusp gives us a horizontal prime divisor. If there is no danger of confusion which
normal model we consider we denote by $\calS_{x_i},\calS_{x},\calS_{y_i}$, etc. the
Zariski closure of $ S_{x_i}, S_{x}, S_{y_i}$, etc.
\end{nota}

\begin{prop}\label{justi2}
Let $S$ be a cusp of $F_N$ and $\calS$ the horizontal divisor obtained by taking the
Zariski closure of $S$ in $\frakF_{N,\frakp}^{min}$. Then $\calS$ only intersects one
component of the geometric special fiber, namely one of the $L_1$, see Figure~\ref{fig:minmodel}. This intersection is transversal.
\end{prop}

\begin{prf}
We use Notation \ref{nota:cusp}. By symmetry, we assume without loss of generality that $S=S_{x_i}$ for some
$i$. If we take the Zariski closure of $S$ in \[
\frakF_{N,\frakp}^{0}=\Proj R[X,Y,Z]/(X^N+Y^N-Z^N)\, , \] we get a horizontal
divisor $\calS^{0}$ which corresponds to the prime ideal $(X,Y-\zeta_N^{i}, Z-1)$. It
intersects the special fiber in the point $P_{x_i}=V_+ ((X,Y-\zeta_N^i, Z-1,  \pi))$. Now
our minimal regular model $\frakF_{N,\frakp}^{min}$ comes with a birational
morphism \begin{equation}\label{eq:22} f: \frakF_{N,\frakp}^{min}\rightarrow
  \frakF_{N,\frakp}^0 \, ;\end{equation} in fact, $f$ is just the composition of the
  blow-ups described in Proposition \ref{prop:blowup1}, Theorem \ref{thm:F_Nreg} and Theorem \ref{thm:F_Nmin}. We have \begin{equation}\label{eq:justi2}
	\frakF_{N,\frakp}^{min}\times_{\Spec R} \Spec \overline{\FF}_p \cdot \calS=
        \deg_{K^{sh}} S=1 \, ,\end{equation} where $K^{sh}=\Frac{R^{sh}}$, see for
        instance \cite[Remark~9.1.31]{Liu}. It follows that
        \[\frakF_{N,\frakp}^{min}\times_{\Spec R} \Spec \overline{\FF}_p \cap \calS\] is
        reduced to a point $P$ and that $P$ belongs to a single irreducible component
        which is of multiplicity one, cf. \cite[Corollary~9.1.32]{Liu}. Furthermore,
        \eqref{eq:justi2} shows that $\calS$ intersects this component transversally, see
        \cite[Proposition~9.1.8]{Liu}. On the other hand, we have $P \in f^{-1}(P_{x_i})$.
        But $f^{-1}(P_{x_i})$ consists of one component $L_{XYZ}$ and $p$ chains of
        components $L_1, L_2, \ldots, L_{(m-1)}$, where a component $L_{(m-1)}$ intersect the
        component $L_{XYZ}$, cf. Figure~\ref{fig:minmodel}. As the only components of
        $f^{-1}(P_{x_i})$ of multiplicity one are the $L_1$'s, $P$ must lie on one of them.
\end{prf}

\begin{rem}\label{rem:urbild}
  In analogy with Proposition~\ref{justi2}, the horizontal divisor that corresponds to
  a cusp $S_{y_i}$ ($S_{z_i}$ resp.) intersects a component $L_1$ that lies in $f^{-1}
  (P_{y_i})$ ($f^{-1} (P_{z_i})$ resp.), where $P_{y_i}=V_+ ((X-\zeta_N^i, Y, Z-1,  \pi))$
  and $P_{z_i}=V_+ ((X-\zeta_N^i,Y+1, Z,  \pi))$, and no other component.
\end{rem}

Since there are $3N$ components $L_1$ and $3N$ cusps it seems plausible that each $L_1$ is
intersected by exactly one horizontal divisor which comes from a cusp. We show in the next
proposition that this is indeed the case.

\begin{prop}
Let $S$ and $S^{\prime}$ be cusps of $F_N$ and denote by $\calS$ and $\calS^{\prime}$ the
associated horizontal divisors of $\mathfrak{F}_{N,\frakp}^{min}$. Suppose that $\calS$
($\calS^{\prime}$, resp.) intersects the component $L$ ($L^{\prime}$, resp.).
Then we have $S=S^{\prime}$ if and only if $L=L^{\prime}$.
\end{prop}
\begin{prf}
  It is clear that $L=L'$ if $S=S'$. Conversely, suppose
 that $S\neq S^{\prime}$, but $L=L^{\prime}$.
According to Remark \ref{rem:urbild} we may assume without loss of generality that
$S=S_{x_i}$ and $S^{\prime}=S_{x_j}$ with $0\leq j<i<N$. The morphism $f$ in
\eqref{eq:22} factors as $f: \mathfrak{F}_{N,\frakp}^{min} \stackrel{f_1}{\rightarrow}
\mathfrak{F}_{N,\frakp}^1\stackrel{f_0}{\rightarrow} \mathfrak{F}_{N,\frakp}^0 $, where
$\mathfrak{F}_{N,\frakp}^1$ is the blow-up of $\mathfrak{F}_{N,\frakp}^0$ along
$V(X^m+Y^m-Z^m, \pi)$. The scheme $\mathfrak{F}_{N,\frakp}^1$ is covered by
$\widetilde{\calX}$ and $\widetilde{\calX}^{\prime}$ (see the beginning of the proof of
Theorem \ref{thm:F_Nmin}) and its special fiber consists of the components $F_m, L_{XYZ},
L_{\gamma_i}$ and $L_\delta$. According to our assumption we must have $\supp f_1
(\calS_{x_i})\cap \supp f_1 (\calS_{x_j})= P$, where $P$ is a closed point which lies
in the special fiber of $\mathfrak{F}_{N,\frakp}^1$; this follows because all the
components $L_i$ are blown down to points by $f_1$. In fact $P$ is a singular point which
lies in the affine open subscheme $\widetilde{\calX}$ defined in Proposition \ref{prop:blowup1}.
By~\eqref{eq:card1} and the proof of Lemma \ref{lem:part1}, all
singular points of $\widetilde{\calX}$ lie in $\blow{1}=\Spec S_1$, so  we can restrict our
attention to this affine open subset. Because $F_m=\bW{1}\pi$ in $S_1$, an easy computation shows that \[ f_1 (\calS_{x_i})|_{\blow{1}}=V\left(X, Y-\zeta_N^i, \bW{1}-\frac{(\zeta_N^{im}-1)}{\pi} \right) \] and
 \[ f_1 (\calS_{x_j})|_{\blow{1}}=V\left(X, Y-\zeta_N^j, \bW{1}-\frac{(\zeta_N^{jm}-1)}{\pi} \right)
 \] (note that $\frac{(\zeta_N^{km}-1)}{\pi} \in R^{\ast}$ or
 $\frac{(\zeta_N^{km}-1)}{\pi}=0$ since $\zeta_N^m$ is a primitive $p$-th root of unity).
 Let $\frakm$ be the maximal ideal of $S_1$ such that $V(\frakm)=P$. Then \[ \zeta_N^i
 -\zeta_N^j=\zeta_N^{j}(\zeta_N^{i-j}-1) \in \frakm \] and since $\pi \in \frakm$, we must
 have $p \nmid i-j$. Indeed, let us assume that $p$ divides $i-j$. Then the order of
 $\zeta_N^{i-j}$ is coprime to $p$ and therefore $\frakm$ contains a natural number
 coprime to $p$, leading to a contradiction. On the other hand, since \[
 \frac{(\zeta_N^{im}-1)}{\pi}-\frac{(\zeta_N^{jm}-1)}{\pi} =
 \frac{\zeta_N^{jm}(\zeta_N^{(i-j)m}-1)}{\pi} \in \frakm \, ,\] we have
 $\zeta_N^{(i-j)m}=1$, hence $p\mid i-j$. This gives us another contradiction and shows that $S=S^{\prime}$. 
\end{prf}

\subsection{Some vertical $\QQ$-divisors and intersections}\label{sec:int}
In this paragraph we define and study some $\QQ$-divisors on 
$\fnpm\times_{\Spec R} \Spec R^{sh}$.
These will be used to compute the geometric contribution in the upper bound given by
Theorem~\ref{thm:keyformula}  and the lower bound $\beta_S$ in Theorem~\ref{thm:km}.
The results are quite technical and the proofs consist mainly of straightforward, but
lengthy calculations.
Recall that $\calT$ denotes the set of irreducible components of the special fiber
$\frakF_{\pi}$ and that
\[
 \frakF_{\pi} = \sum_{\calC \in \calT}d_\calC\,\calC\,,
\]
where the components $\calC \in \calT$ and their multiplicities $d_\calC$ are given in
Figure~\ref{fig:minmodel} and Table~\ref{tabular:quant}. 

\begin{nota}\label{nota:comp} We use the notation from Theorem \ref{thm:F_Nmin}. Let us
  fix a cusp $S$ and a corresponding horizontal divisor $\calS$. We know that $\calS$ 
  intersects precisely one of the component of the special fiber; in fact it must be one of the
  components $L_1$ (Proposition \ref{justi2}). In  the geometric special fiber
  $\mathfrak{F}_\pi$ there are $3m$ components $L_{XYZ}$. To distinguish between
  these components we will number them and denote by $L^{(i)}$ the $i$-th one of the
  $L_{XYZ}$. Now for each component $L^{(i)}$ there are $p$ chains of components $L_1,
  L_2, \ldots L_{(m-1)}$, where the $L_{(m-1)}$ intersect $L^{(i)}$. Again, we will number
  these chains. We denote the components of the chains by $L_{j,k}^{(i)}$, where the first
  subscript $j$ indicates that it is one of the components $L_j$, the second subscript
  $k$ means that it is a component of the $k$-th chain, and the
  superscript $(i)$ indicates that the chain is attached to $L^{(i)}$. In the same way we
  proceed with the components $L_\gamma$ and $L_\delta$. We will number them and denote
  them by $L_\gamma^{(i)}$ and $L_\delta^{(i)}$. The components $L_{\gamma, j}$ will be
  denoted by $L_{\gamma,j}^{(i)}$, where the superscript $i$ indicates that
  $L_{\gamma,j}^{(i)}$ intersects $L_{\gamma}^{(i)}$.  Without loss of generality we
  assume that we fixed this numbering so that $\calS$ intersects the component $L_{1,1}^{(1)}$.
\end{nota}

We now define the following vertical $\QQ$-divisors on $\frakF_\pi$:
\begin{align*}
  V_{F_m} = &\, \frac{p-2}{2g-2}F_m\\
  V_{L_{\delta}^{(i)}} = &\, V_{F_m} +\frac{1}{p}L_{\delta}^{(i)},\quad 1 \le i \le
  m^2(p-3)-2m\varrho\\
  V_{L_{\gamma}^{(i)}} = &\, V_{F_m}
  +\frac{1}{p}L_{\gamma}^{(i)}+\sum^p_{j=1}\frac{1}{2p}L_{\gamma,j}^{(i)},\quad 1\le i \le
  m\varrho\\
  V_{L_{\gamma,s}^{(i)}} = &\, V_{F_m}
  +\frac{1}{p}L_{\gamma}^{(i)}+\sum^p_{j=1}\frac{1}{2p}L_{\gamma,j}^{(i)}+\frac{1}{2}L_{\gamma,s}^{(i)},\quad
  1\le i \le m\varrho,\,1 \le s \le p\\
  V_{L^{(i)}} = &\, V_{F_m}
  +\frac{1}{p}L^{(i)}+\sum^{m-1}_{j=1}\sum^p_{k=1}\frac{j}{N}L_{j,k}^{(i)},\quad 1\le i \le 3m \\
  V_{L_{r,s}^{(i)}} = &\, V_{F_m}
  +\frac{r}{p}L^{(i)}+\sum^{m-1}_{j=1}\sum^p_{k=1}\frac{jr}{N}L_{j,k}^{(i)}+\sum^{r-1}_{j=1}\frac{j(m-r)}{m}L_{j,s}^{(i)}\\&\,\quad+\sum^{m-1}_{j=r}\frac{r(m-j)}{m}L_{j,s}^{(i)},\quad
  1\le i \le 3m, \,1 \le r \le m-1,\,1 \le s \le p 
\end{align*}

Recall that if $\calC$ is an irreducible component of $\mathfrak{F}_\pi$, then
  we have $a_\calC = (\calK\cdot \calC)$ by the 
  adjunction formula (Theorem~\ref{thm:kan}), where
$a_\calC = -\calC^2+2p_a(\calC) -2$ and $\calK$ is a canonical $\QQ$-divisor of
$\fnpm\times_{\Spec(R)} \Spec(R^{sh})$.

\begin{lemma}\label{lem:v_c}
  Let $\calD \in \calT$ be an irreducible component of $\mathfrak{F}_\pi$. Then we have
  \[
    (V_{\calD} \cdot \calC) = \frac{a_{\calC}}{2g-2} -
    \frac{\delta_{\calD,\calC}}{d_\calC}
  \]
  for all $\calC \in \calT$, where $\delta$ is the Kronecker delta on $\calT$.
\end{lemma}
\begin{proof}
  This can be verified by a straightforward computation using Theorem~\ref{thm:F_Nmin}.
\end{proof}
Next we compute the self-intersections of the $\QQ$-divisors $V_{\calD}$.
Let us denote
\[
  \lambda = -\left(\frac{m(p-2)}{2(g-1)}\right)^2\quad \textrm{and} \quad \nu =\frac{p-2}{p(g-1)}.
\]

\begin{lemma}\label{lem:v_cints}
We have
\begin{align*}
  V_{F_m}^2 &= \lambda\\
  V_{L_{\delta}^{(i)}}^2 &= \lambda+\nu -\frac{1}{p}\\
  V_{L_{\gamma}^{(i)}}^2 &= \lambda+\nu -\frac{1}{2p}\\
  V_{L_{\gamma,s}^{(i)}}^2 &= \lambda+\nu -\frac{1+p}{2p}\\
  V_{L^{(i)}}^2 &= \lambda+\nu -\frac{1}{N}\\
  V_{L_{r,s}^{(i)}}^2 &= \lambda+r\nu 
  -\frac{r+N-rp}{N}\, .\\
\end{align*}
\end{lemma}
\begin{proof}
  We obviously have $ V_{F_m}^2 = \lambda$.
  For the other components $\calD \in \calT$, we can write $V_{\calD} = V_{F_m} +
  W_{\calD}$ and compute 
  \[
    V_{\calD}^2 = \lambda + W_{\calD}^2 + p\nu(F_m\cdot W_{\calD})\,.
  \]  
  Alternatively, we can write $V_{\calD} = \sum_{\calC \in \calT}
  r_\calC\,\calC$ and use Lemma~\ref{lem:v_c}, which implies
  \begin{equation}\label{vc0sq}
    V_{\calD}^2 =\left(V_{\calD}\cdot\sum_{\calC \in \calT}r_{\calC}\, \calC\right)
                  =\sum_{\calC \in \calT}r_{\calC}\,\left(\frac{a_{\calC}}{2g-2} -
                  \frac{\delta_{\calD\calC}}{d_\calC}\right)\, .
  \end{equation}
  Either one of these formulas leads to a straightforward proof of the assertion.
\end{proof}

Recall that we have fixed a cusp $S$ whose Zariski closure
$\calS$ in $\fnpm$ intersects the component $L^{(1)}_{1,1}$,
and no other $\calC \in \calT$.
Setting
\begin{equation}\label{eq:v_sdef}
 V_{S,\frakp} =   V_S = V_{L^{(1)}_{1,1}}\, ,
\end{equation}
Lemma~\ref{lem:v_c} implies
\begin{equation}\label{eq:v_sprop}
  (\calS + V_S) \cdot \calC = \frac{a_\calC}{2g-2}
\end{equation}
for all $\calC \in \calT$.
Note that we have
\begin{equation}\label{eq:v_salt}
  V_S = 2p\nu F_m + \frac{1}{p}L^{(1)} + \sum^{m-1}_{j=1}\sum^p_{k=1}\mu_{j,k}L_{j,k}^{(1)}\,,
\end{equation}
where 
\[\mu_{j,1}=\frac{j-jp+N}{N} \quad\mbox{ and }\quad
\mu_{j,k}= \frac{j}{N}  \mbox{ for }  k\neq 1 \, .
\]
The $\QQ$-divisor $V_{S}$ will be play a crucial part in Section~\ref{sec:upper}. 
On the one hand, it will be used to construct the divisors $\calF_j$ (defined
in~\eqref{eq:F_i}) whose self-intersections appear in Theorem~\ref{thm:keyformula}.
On the other hand, the lower bound $\beta_S$ from Theorem~\ref{thm:km} is defined using
$V_S$.

We start by analyzing the intersections of $V_S$ with the $\QQ$-divisors $V_\calC$ for $\calC \in \calT$.
\begin{lemma}\label{lem:v_dints}
  We have 
\begin{align*}
  (V_S \cdot V_{F_m}) &= \lambda+\frac{1}{2}\nu \\
(V_S \cdot V_{L_{\delta}^{(i)}}) &= \lambda+\nu \\
(V_S \cdot V_{L_{\gamma}^{(i)}}) &= \lambda+\nu \\
(V_S \cdot V_{L_{\gamma,s}^{(i)}}) &= \lambda+\nu \\
(V_S \cdot V_{L^{(i)}}) &= \lambda+\nu - \frac{\delta_{1i}}{N} \\
(V_S \cdot V_{L_{r,s}^{(i)}}) &= \lambda + \frac{r+1}{2}\nu - \frac{r\delta_{1i}}{N}-
  \frac{(m-u)\delta_{1i}\delta_{1s}}{m}\, , \\
\end{align*}
where $\delta$ is the Kronecker delta on $\{1,\ldots,3m\}$.
\end{lemma}
\begin{proof}
  The proof is similar to the proof of Lemma~\ref{lem:v_cints}.
  Namely, if $V_{\calD} = \sum_{\calC \in \calT} r_\calC\,\calC$, then Lemma~\ref{lem:v_c} implies
  \[
    \left(V_S \cdot V_{\calD}\right) 
    =\sum_{\calC \in \calT}r_{\calC}\,\left(\frac{a_{\calC}}{2g-2} - \frac{\delta_{L_{1,1}^{(1)},\calC}}{d_\calC}\right)\, .
  \]
  Using this, the proof consists of elementary computations.
\end{proof}
We now use the vertical $\QQ$-divisors $V_{\calC}$ to define another vertical $\QQ$-divisor
\[
  U_{S,\frakp} =  U_S  = \sum_{\calC \in \calT} d_\calC(2(V_\calC \cdot V_S)
-V_S^2)\,\calC \,- (\lambda+\mu)\mathfrak{F}_\pi\, .
\]

\begin{lemma}\label{lem:udform}
  We have
\begin{align*}
  U_S =\,& 
  \sum^{m^2(p-3)-2m\varrho}_{i=1}\frac{1}{p}L_{\delta}^{(i)} +
  \sum_{i=1}^{m\varrho}\frac{1}{p}L_\gamma^{(i)}
 +\sum^{m\varrho}_{i=1}\sum^{p}_{j=1}\frac{1+p}{p}L_{\gamma,j}^{(i)}
   +\sum^{3m}_{i=1}\frac{1}{p}L^{(i)}
   -\frac{2}{p}L^{(1)}\\
   &+\sum^{3m}_{i=1}\sum^{m-1}_{j=1}\sum^{p}_{k=1}j\mu_{j,1}L^{(i)}_{j,k} 
   -\sum^{m-1}_{j=1}\sum^{p}_{k=1}\frac{2j}{N}L^{(1)}_{j,k}
   -\sum^{m-1}_{j=1}\frac{2(m-j)}{m}L^{(1)}_{j,1}\,.
\end{align*}
\end{lemma}
\begin{proof}
This is a simple computation using Lemma~\ref{lem:v_cints} and Lemma~\ref{lem:v_dints}.
\end{proof}

As a corollary, we get the following result on the intersection multiplicities between $U_{S}$ and the
components $\calC \in \calT$.

\begin{lemma}\label{lem:rel_semipos}
  If $\calC$ is an irreducible component of $\mathfrak{F}_\pi$, then we have
  \[
      a_\calC+2(\calS\cdot \calC) - (U_S\cdot \calC) \ge 0\, .
  \]
\end{lemma}
\begin{proof}
  We only show the claim for $\calC = L_{\delta}^{(i)}$.
  Using Lemma~\ref{lem:udform}, we find
  \[
    \left(U_S \cdot L_{\delta}^{(i)}\right) = \frac{1}{p}(L_{\delta}^{(i)})^2 = -1
\]
  and hence
  \[
    a_{L_{\delta}^{(i)}}+ (\calS\cdot L_{\delta}^{(i)}) - (U_S\cdot L_{\delta}^{(i)}) = p-1 \ge 0\, .
  \]
  The other cases are similar and are left to the reader. 
\end{proof}

Let us define
\[
   \beta_{S,\frakp} = \beta_S = \frac{1-g}{g}(2V_S+U_S)^2 + 2(\calK \cdot U_S)\, ,
\]
where $\calK$ is a canonical $\QQ$-divisor of $\fnpm\times_{\Spec(R)} \Spec(R^{sh})$.
Summing up all $\beta_{S,\frakp}$ as $\frakp$ runs through the bad primes of $\calO_K$, we
will get a lower bound for $\omm^2$ in~\S\ref{sec:lower_fermat} using Theorem~\ref{thm:km}.

\begin{prop}\label{prop:beta_S}
We have 
  \[
\beta_S = N(\lambda+\nu)\left(\frac{N(\lambda+\nu)(g-1)}{g}+4m-6\right)\, .
  \]
\end{prop}
\begin{proof}
  Applying Lemma~\ref{lem:udform}, we see that
\begin{align*}
  2V_S+U_S =\,&
  \sum^{m^2(p-3)-2m\varrho}_{i=1}\frac{1}{p}L_{\delta^{(i)}} +
  \sum_{i=1}^{m\varrho}\frac{1}{p}L_\gamma^{(i)}
  +\sum^{m\varrho}_{i=1}\sum^{p}_{j=1}\frac{1+p}{p}L_{\gamma,j}^{(i)}
   +\sum^{3m}_{i=1}\frac{1}{p}L^{(i)}_{XYZ}\\
   &+\sum^{3m}_{i=1}\sum^{m-1}_{j=1}\sum^{p}_{k=1}\mu_{j,1}L^{(i)}_{j,k} \, .
\end{align*}
A simple computation shows 
\begin{equation}\label{vsussq}
  (2V_S+U_S)^2 = -(N(\lambda+\nu))^2\, .
\end{equation}
Using the adjunction formula, it is easy to see that
\begin{equation}\label{kus}
  (\calK \cdot U_S) =(2m-3)N(\lambda+\nu)\, .
\end{equation}
The result follows from~\eqref{vsussq} and~\eqref{kus}.
\end{proof}
 
\begin{rem}\label{rk:v_sindep}
  Suppose that $S$ is a cusp whose Zariski closure $\calS$ intersects $L^{(i)}_{1,k}$,
  where $(i,k) \ne (1,1)$. Then Lemma~\ref{lem:rel_semipos} and
  Proposition~\ref{prop:beta_S} remain valid (with the obvious index modifications); the proofs are entirely analogous.
\end{rem}

It remains to compute local versions of the divisors $\calG_j$, defined in~\eqref{eq:G_i}.
By Theorem~\ref{thm:keyformula}, these are needed for the upper bound for $\omm^2$.
As $\mathfrak{F}_{N,\frakp}^{min}$ is constructed using a sequence of blow-ups, the morphism $\belyi :F_N \rightarrow \PP^1$ in \eqref{derMor} extends to a morphism 
\[\belyi: \mathfrak{F}_{N,\frakp}^{min} \rightarrow \PP^1_R\, . \] 
For our applications (see Section~\ref{sec:upper}) we need to construct a divisor
of $\frakF_{N,\frakp}^{min}$ whose associated line bundle is isomorphic
to the pullback of the twist $\calO_{\PP^1_R}(1)$  by $\belyi$.

We set \begin{equation}\label{eq:calG_xp}
  \calG_{S, \frakp}  = \calG_S = \sum_{j=1}^{m-1} \sum_{k=1}^p \mu_{j,k} L_{j,k}^{(1)} +
  \mu L_{1,1}^{(1)} \, , \end{equation}

\begin{lemma}\label{lem:D_xF_N}
  Let \[ \calE_{S}=\calS + \calG_{S} \, ,\]
 where $\calG_S$ is the vertical $\QQ$-divisor in \eqref{eq:calG_xp}. Then
$\calE_S$ is a $\QQ$-divisor of $\mathfrak{F}_{N,\frakp}^{min}$ which is associated to
$\left(\belyi^{\ast}\calO_{\PP^1_R} (1) \right)^{\otimes \frac{1}{N^2}}$.
\end{lemma}
\begin{prf}
We can show that $N^2 S$ is associated to
$\belyi^{\ast} \calO_{\PP^1_K} (1)$, where $K$ is the fraction field of $R$, using
arguments analogous to those employed in~\cite[Lemma~7.3]{CK}. Since \[
\belyi^{\ast} \calO_{\PP_{R}^1} (1)|_{F_N} \cong \belyi^{\ast}\calO_{\PP_K^1} (1)\, ,\]  it
is clear that
there is a $\QQ$-divisor of the form $\calE_S=\calS+\calG_{S}$, with
a vertical $\QQ$-divisor $\calG_S$, such that $\calE_S$ is associated to 
$\left(\belyi^{\ast}\calO_{\PP^1_R} (1) \right)^{\otimes \frac{1}{N^2}}$.
The $\QQ$-divisor $\calE_S$ has to satisfy the equations \begin{equation}\label{eq:D_xp1}
(N^2 \calE_S \cdot \calC) =0 \end{equation} for all components $\calC$ which are
different from $F_m$ (see e.g. \cite[Theorem 9.2.12]{Liu}), and \begin{equation}\label{eq:D_xp2}
N^2=\left(N^2 \calE_S \cdot \frakF_{N,\frakp}^{min} \times_{\Spec R} \Spec \FF_p\right)=
(N^2\calE_S \cdot pF_m)\, , \end{equation} see~\cite[Remark~9.1.131]{Liu}. One can
use the quantities computed in Theorem~\ref{thm:F_Nmin} to verify that our choice of
$\calG_S$ in~\eqref{eq:calG_xp} indeed satisfies the equations~\eqref{eq:D_xp1} and~\eqref{eq:D_xp2}.
\end{prf}

\begin{prop}\label{prop:gp_int}
  We have \[ \calG_S^2=-\frac{N-p+1}{N} \, .\]
\end{prop}
\begin{prf}
  Note that by~\eqref{eq:v_salt}, we have  $\calG_S = V_{S} - V_{F_m}$. 
  Thus the result follows from Lemma~\ref{lem:v_cints} and Lemma~\ref{lem:v_dints}.
\end{prf}

\begin{rem}\label{rk:g_sindep}
  Suppose that $S$ is a cusp whose Zariski closure $\calS$ intersects $L^{(i)}_{1,k}$,
  where $(i,k) \ne (1,1)$. Then analogues of Lemma~\ref{lem:D_xF_N} and
  Proposition~\ref{prop:gp_int} for $S$ can be proved in an similar way.
\end{rem}

\section{Bounds for $\omm^2$}\label{sec:upper}
In this section we compute upper and lower bounds for the arithmetic self-intersection $\omm^2$ of the
dualizing sheaf on the minimal regular model $\fnm$ over $\Spec(\ZZ_N)$ of the Fermat curve $F_N$.

\subsection{An upper bound for $\omm^2$}\label{sec:upper_fermat}
We want to apply Theorem~\ref{thm:keyformula} to find an upper bound for $\omm^2$.
The morphism $\belyi :F_N \rightarrow \PP^1$ from~\eqref{derMor} is unramified outside
$0, 1, \infty$ and extends to a morphism \[
\belyi: \mathfrak{F}_N^{min} \rightarrow \PP^1_{\ZZ [\zeta_N]}  \,, \]
since the minimal regular model $\mathfrak{F}_N^{min}$ can be constructed by a sequence of
blow-ups, see Section~\ref{sec:curilla}.
We will apply Theorem~\ref{thm:keyformula} with $\belyi: \mathfrak{F}_N^{min} \rightarrow \PP^1_{\ZZ[\zeta_N]}$.  
To compute the geometric contribution, we construct $\QQ$-divisors $\calF_j$ and $\calG_j$ as in
Section~\ref{sec:kuehn}, using the local results from \S\ref{sec:int}.
Recall that the cusps on $F_N$ are the points which are mapped to $0\,,1$ or $\infty$ by $\belyi$ and that a
divisor on $F_N$ is called cuspidal if its support consists entirely of cusps.

We first construct the $\QQ$-divisors $\calF_j$.

\begin{thm}[Rohrlich]\label{thm:Rohrlich}
The group of cuspidal divisors on $F_N$ modulo the group of principal cuspidal divisors is
a torsion group.
\end{thm}
\begin{prf}
The statement follows from \cite[Theorem 1]{Roh}.
\end{prf}

\begin{coro}\label{coro:Rohrlich}
  Let $S\in F_N(\QQ(\zeta_N))$ be a cusp. Then $(2g-2)S$ is a canonical divisor. 
\end{coro}
\begin{prf} 
The corollary follows from Theorem~\ref{thm:Rohrlich}, because
  the Hurwitz formula implies that there exists a canonical divisor with support in the cusps. 
\end{prf}

\begin{prop}\label{prop:fj}
  Let $S_j\in F_N$ be a cusp and let $\calS_j\in
  \Div(\fnm)$ be its Zariski closure. Set 
\[ \calF_j= \sum_{\frakp \, \,  {\rm bad}}V_{S_j,\frakp}\, ,\]  
  where $V_{S_j,\frakp}$ is the vertical $\QQ$-divisor supported in the special fiber above $\frakp$
defined in~\eqref{eq:v_sdef}, viewed as a $\QQ$-divisor on $\fnm$.
Then
\begin{enumerate}[(i)]
  \item $(2g-2)(\calS_j + \calF_j)$ is a canonical $\QQ$-divisor on $\fnm$;
  \item $\calF_j$ satisfies \eqref{eq:F_i}. 
\end{enumerate}
\end{prop}
\begin{proof}
  It is clear that (ii) follows from (i).
By Corollary~\ref{coro:Rohrlich}, the divisor $(2g-2)S_j$ is a canonical divisor on $F_N$. 
Hence, by~\cite[Proposition~2.5]{CK}, a $\QQ$-divisor $\calK $ on $\fnm$ of the form 
\[
  \calK=(2g-2)\calS_j + \calV \, ,
\]
where $\calV$ is a vertical $\QQ$-divisor, is canonical if and only if $\calK$ satisfies the adjunction formula (Theorem~\ref{thm:kan}).
By~\eqref{eq:v_sprop}, the $\QQ$-divisor
\[
  \calK = (2g-2)(\calS_j + \calF_j)
\]
satisfies the adjunction formula, so (i) follows. 
\end{proof}

We now find, for cusps $S_j$ above $\infty$,  $\QQ$-divisors $\calG_j$ such
that~\eqref{eq:G_i} is satisfied. 
To this end, we use the vertical $\QQ$-divisors $\calG_{S_j,\frakp}$, see~\eqref{eq:calG_xp}
and Remark~\ref{rk:g_sindep}.

\begin{lemma}\label{lem:help}
 Let $S_j\in F_N$ be a cusp above $\infty$ with Zariski closure $\calS_j \in \Div(\fnm)$.
 Then the $\QQ$-divisor \[ \calE_{S_j}=\calS_j + \sum_{\frakp \,\, {\rm bad}}
 \calG_{S_j,\frakp}\]  is associated
with $(\belyi^{\ast} \calO_{\PP^1_{\ZZ[\zeta_N]}} (1))^{\otimes \frac{1}{N^2}}$, where we
view each $\calG_{S_j,\frakp}$ as a $\QQ$-divisor on $\fnm$.
\end{lemma}
\begin{prf}
  We can assume that the $\QQ$-divisor we are looking for is of the form
  $\calE_{S_j}=\calS_j+\calG$, where $\calG$ is a 
vertical $\QQ$-divisor  with support in the bad fibers. If $\frakp$ is a prime of bad
reduction above $p$ and $\calC$ is an irreducible component of the special fiber above
$\frakp$ which is different from the component $F_{N/p}$, then the $\QQ$-divisor $\calE_{S_j}$ has to satisfy \[
(N^2\calE_{S_j} \cdot \calC )=0 \, .\] Furthermore, $\calE_{S_j}$ has to satisfy \[ 
  N^2 =\left(N^2 \calE_{S_j} \cdot \frakF_{N}^{min} \times_{\Spec \ZZ[\zeta_N]} \Spec
  \FF_p\right)= \left(N^2\calE_{S_j} \cdot pF_{N/p}\right) \, .\] 
  On the other hand, if we take $\calG=\sum_{\frakp\,\, {\rm bad}}
\calG_{S_j,\frakp}$, then these equations are satisfied, because a component $\calC$ which
belongs to the fiber above $\frakp$ only intersects $\calG_{S_j,\frakp}$. It follows that our choice of $\calG$ is valid.
\end{prf}

\begin{coro}\label{cor:sumoverfrakp}
Let $S_j$ be a cusp which lies above the branch point $\infty$. 
Let us set \[
\calG_j=\sum_{\frakp \, \,  {\rm bad}} \calG_{S_j,\frakp} \, .\] Then 
$\calG_j$ satisfies \eqref{eq:G_i}.
\end{coro}
\begin{prf}
The Zariski closure $\overline{\infty}$ of $\infty$ in $\PP_{\ZZ[\zeta_N]}^1$ is
associated to $\calO_{\PP_{\ZZ[\zeta_N]}^1}(1)$. 
Because $S_j$ lies above the branch point $\infty$, 
Lemma~\ref{lem:help} implies that~\eqref{eq:G_i} is satisfied for the section $s = \belyi^*(1) \in
\belyi^*\calO(\infty)$.
\end{prf}

\begin{lemma}\label{lem:F_NGF} 
Let $S_j$ be a cusp above $\infty$,
let $p|N$ be a prime and let $\frakp$ be a prime above $p$.
Then the self-intersections
$V_{S_j,p}^2\colonequals V_{S_j,\frakp}^2$ and $\calG_p^2\colonequals \calG_{S_j,\frakp}^2$ are independent of
$\frakp$. Furthermore, we have 
\[ \overline{\calO} (\calF_j)^2 =\sum_{p|N} \varphi(N)/\varphi(p) V_{S_j,p}^2 \log p
\] and  \[ \overline{\calO} (\calG_j)^2 =\sum_{p|N} \varphi(N)/\varphi(p) \calG_{S_j,p}^2
\log p \, ,\] 
\end{lemma}
\begin{prf}
For prime ideals of $\ZZ[\zeta_N]$ above the same prime number $p$, the corresponding
special fibers of $\mathfrak{F}_N^{min}$ are isomorphic, proving the first statement.

We have \[
\overline{\calO} (\calF_j)^2=\sum_{\frakp \,\, {\rm bad}}\overline{\calO}
(V_{S_j,\frakp})^2=\sum_{p|N} \sum_{\genfrac{}{}{0pt}{}{\frakp \,\, {\rm bad}}{\frakp \cap
\ZZ=(p)} }\overline{\calO} (V_{S_j,\frakp})^2\, ,\] with
$\overline{\calO}(V_{S_j,\frakp})^2=\calF_{\frakp}^2 \log \Nm (\frakp)$
by~\eqref{vertintsum}. For each prime $p$
let us denote by $r_p$ the number of prime ideals of $\ZZ [\zeta_N]$ that lie above $p$.
Since $\QQ (\zeta_N)/ \QQ$ is a Galois extension, all the inertia degrees and
ramification indices of the prime ideals over $p$ are the same (we denote them by
$f_p$ and $e_p$, respectively), and we get the equation
$\varphi(N) = [\QQ (\zeta_N): \QQ]=e_pf_pr_p$. Because $e_p=\varphi(p)$,
we have \[
r_p \log \Nm (\frakp)=\varphi(N)/\varphi(p) \log (p)\,  \]
for a prime ideal $\frakp$ above $p$,
Hence it follows that \[\sum_{\genfrac{}{}{0pt}{}{\frakp \,\, {\rm bad}}{\frakp \cap
\ZZ=(p)} }\overline{\calO} (V_{S_j,\frakp})^2= r_p V_{S_j,\frakp}^2 \log \Nm (\frakp)=
\varphi(N)/\varphi(p) V_{S_j,\frakp}^2 \log p  \, .\] 
Summing up over all prime numbers $p$ with $p|N$, we obtain the formula for $\overline{\calO}(\calF_j)^2$. 
The claimed formula for $\overline{\calO}(\calG_j)^2$ can be verified in a similar way.
\end{prf}

We now prove the main result of this section.

\begin{thm}\label{thm:fermat}
  Let $N>0$ be an odd squarefree integer with at least two prime factors, and let $\mathfrak{F}_{N}^{min}$ be the minimal regular model of the Fermat curve
  $F_N$ over $\Spec \ZZ[\zeta_N]$. Then the arithmetic self-intersection number
  of its dualizing sheaf equipped with the Arakelov metric satisfies
\jot3mm 
\begin{align*} 
\overline{\omega}_{\mathfrak{F}_N^{min}}^2 \le&\; (2 g-2) 
  [\QQ(\zeta_N):\QQ] (   \kappa_1 \log N + \kappa_2)\\ 
   & \; + (2 g-2) \sum_{p|N}\frac{\varphi (N)}{\varphi (p)} \left( \frac{3N^2-2Np-10N+6p-6-4\left(\frac{N}{p}\right)^2+12 \left(\frac{N}{p}\right)}{N(N-3)}\right)  \log p \, ,
\end{align*}
\jot1mm
where $\kappa_1, \kappa_2\in \RR$ are  positive constants independent of $N$.
\end{thm}
\begin{prf}
  The morphism  $\belyi : \mathfrak{F}_N^{min} \rightarrow \PP_{\ZZ [\zeta_N]}^1$ is a
  morphism of arithmetic surfaces which satisfies the requirements
  of Theorem \ref{thm:keyformula}. 
  We have $\deg \belyi=N^2$ and $\belyi^{\ast}\infty=\sum_{j=1}^N NS_j$, hence $b_j=b_{\max}=N$. It follows that in our case the formula
  \eqref{eq:def-ap} of Theorem \ref{thm:keyformula}
  becomes 
  \jot3mm 
  \begin{align*} 
  \sum_{\frakp \,\,{\rm bad}} a_\frakp
    \log\Nm(\frakp) &= -2g \overline{\calO} (\calG_j)^2+(2g-2)\overline{\calO} (\calF_j)^2 \\
                    &= \sum_{p|N} \frac{\varphi (N)}{\varphi (p)}\left(-2g\calG_{S_j,p}^2 + (2g-2)V_{S_j,p}^2\right) \log p \\
    &= \sum_{p|N}\frac{\varphi (N)}{\varphi (p)} \left( \frac{3N^2-2Np-10N+6p-6-4\left(\frac{N}{p}\right)^2+12 \left(\frac{N}{p}\right)}{N(N-3)}\right)  \log p \, ,
    \end{align*}
     \jot1mm 
     where we used Lemma \ref{lem:F_NGF} for the second equality.
     The final equality follows from Lemma~\ref{lem:v_cints} 
     and Proposition~\ref{prop:gp_int}.
  \end{prf}

For the proof of Theorem~\ref{thm:intro_upper} from the introduction, we also need the
following simple fact.
\begin{lemma}\label{L:easy}
  We have
  \[
    \sum_{p \mid N} \frac{\log p}{p-1} \le \calO(\log\log N)
  \]
  for $N \in \mathbb{N}$ odd and squarefree.
\end{lemma}
\begin{proof}
  We bound $\sum_{1<n \le x} \frac{\log p_n}{p_n-1}$, where $p_n$ is the $n$-th prime. 
  It is well known that
  \[
    n \log n < p_n < n \log n + n \log\log n\,,
  \]
  for $n \ge 6$,  so 
  \[
  \frac{\log p_n}{p_n-1} < \frac{\log n}{n \log n -1}  + \frac{\log(\log n + \log\log
  n)}{n \log n -1} \le \frac{4}{n} \,.
  \]
  follows for $n \ge 6$, implying
  \[
\sum_{1<n \le x} \frac{\log p_n}{p_n-1} \le 4 \log x + c\,,
  \]
  where $c$ is a constant independent of $x$.
  Now the number of prime divisors of $N$ is of order $\calO(\log N/\log\log N)$
  by~\cite[Chapter~22]{HW}, so the result follows.
\end{proof}
    
  We can now deduce Theorem~\ref{thm:intro_upper} 
  from Theorem~\ref{thm:fermat} and Lemma~\ref{L:easy}.
  \vspace{1mm}
\begin{Prf}{Theorem \ref{thm:intro_upper}}
We have
 \jot3mm
  \begin{align*}
\sum_{\frakp \,\,{\rm bad}} a_\frakp
    \log\Nm(\frakp) &= \sum_{p|N}\frac{\varphi (N)}{\varphi (p)} \left( \frac{3N^2-2Np-10N+6p-6-4\left(\frac{N}{p}\right)^2+12 \left(\frac{N}{p}\right)}{N(N-3)}\right)  \log p \notag \\ 
    &\leq  \sum_{p|N} \frac{\varphi (N)}{\varphi (p)} \frac{3N}{N-3} \log p 
\leq  \frac{15}{4} \varphi(N)\sum_{p|N} \frac{\log p}{ p-1} =
\varphi(N)\calO(\log\log N)
\end{align*}
\jot1mm for the geometric contribution by Lemma~\ref{L:easy}.

The analytic contribution is \[ 
 \varphi (N) (\kappa_1 \log N +\kappa_2)= \varphi (N) \kappa_1 \log N + \calO (\varphi
 (N)) \, .\] 
 Setting $\kappa = \kappa_1$, we find
 \[
 \overline{\omega}_{\mathfrak{F}_N^{min} }^2 \leq (2g-2)\kappa\varphi(N)\log N +
 \calO(g\varphi(N)\log\log N)\,,
 \]
 which is the statement of Theorem~\ref{thm:intro_upper}.
 \end{Prf}

\subsection{A lower bound for $\omm^2$}\label{sec:lower_fermat}
In order to use Theorem~\ref{thm:km} to obtain a lower bound for $\omm^2$, we need to find
a suitable rational point $S \in F_N(\QQ(\zeta_N))$ such that properties (i) and (ii) of
Theorem~\ref{thm:km} are satisfied.

Let $S$ be one of the cusps of $F_N$. We use the notation of Section~\ref{sec:km}.
Recall that, for a prime $\frakp|N$ of $\ZZ[\zeta_N]$, we defined vertical $\QQ$-divisors
$V_{S,\frakp}$ and $U_{S,\frakp}$, and gave a formula for 
 \[
   \beta_{S,\frakp} = \frac{1-g}{g}(2V_{S,\frakp}+U_{S, \frakp})^2 + 2(\calK_\frakp \cdot
   U_{S,\frakp})
\]
in Proposition~\ref{prop:beta_S}, where $\calK_\frakp$ is a canonical $\QQ$-divisor on $\fnpm$.
\begin{lemma}\label{lem:beta_form}
  For a prime $p|N$ and a prime $\frakp$ above $p$, the numbers $\beta_{S,p} \colonequals 
  \beta_{S,\frakp}$ are independent of $\frakp$.
Furthermore, we have 
\[ \beta_S =\sum_{p|N} \frac{\varphi(N)}{\varphi(p)} \beta_{S,p} \log p\, .
\] 
\end{lemma}
\begin{prf}
Since all special fibers above primes dividing a prime number $p$ are isomorphic, the
first statement follows.

To prove the second statement, note that
\[
V_S = \sum_{p|N} \sum_{\genfrac{}{}{0pt}{}{\frakp \,\, {\rm bad}}{\frakp \cap
\ZZ=(p)} }V_{S,\frakp}\, 
\]
satisfies~\eqref{eq:vsglob} and that we have
\[
  U_S = \sum_{p|N} \sum_{\genfrac{}{}{0pt}{}{\frakp \,\, {\rm bad}}{\frakp \cap
\ZZ=(p)} }U_{S,\frakp}\,, 
\]
yielding
\[
  \beta_S = \sum_{p|N} \sum_{\genfrac{}{}{0pt}{}{\frakp \,\, {\rm bad}}{\frakp \cap
\ZZ=(p)} }\beta_{S,\frakp}\log\Nm(\frakp)\,.
\]
Now the second statement follows as in the proof of Lemma~\ref{lem:F_NGF}.
\end{prf}

We now prove the main result of this section.
\begin{thm}\label{thm:lower}
  Let $N>0$ be an odd squarefree integer with at least two prime factors, and let $\mathfrak{F}_{N}^{min}$ be the minimal regular model of the Fermat curve
  $F_N$ over $\Spec \ZZ[\zeta_N]$. Then 
  we have
  \[
    \om_{\mathfrak{F}_{N}^{min}}^2 \ge \varphi(N)\sum_{p|N}\frac{\alpha(N,p)(Np + 2N - 6p)(p - 2)}{(N - 1)(N
    - 2)(N - 3)^3p^4(p-1)}\log p\, ,
  \]
  where 
  \[
    \alpha(N,p) = 4N^4p - 6N^3p^2 - 24N^3p + 37N^2p^2 + 44N^2p - 72Np^2 - 4N^2 - 12Np +
  36p^2\, .
  \]
\end{thm}
\begin{proof}
Corollary~\ref{coro:Rohrlich} implies that $(2g-2)S$ is a canonical divisor on $F_N$. 
By Lemma~\ref{lem:rel_semipos}, we see that~\eqref{eq:rel_semipos} is satisfied for all
irreducible vertical components, so that Theorem~\ref{thm:km} is applicable.   
Therefore we obtain the lower bound
\[
 \om_{\mathfrak{F}_{N}^{min}}^2  \ge  \sum_{p\mid N}\frac{\varphi(N)}{\varphi(p)}\beta_{S,p}\log p
\]
from Lemma~\ref{lem:beta_form}.
By Proposition~\ref{prop:beta_S}, we have
\[
  \beta_{S,p} = \frac{ \alpha(N,p)(Np + 2N - 6p)(p - 2)}{(N - 1)(N - 2)(N - 3)^3p^4}\, ,
\]
which proves the result.
\end{proof}

\begin{proof}[Proof of Theorem~\ref{thm:intro_lower}]
  Let $N$ be odd, composite and squarefree and let $p$ be a prime dividing $N$.
  From $p \le \frac{N}{3}$ we get
  \[
   Np + 2N - 6p \ge Np\,.   
  \]
  Moreover, using $N \ge 15$ we find
\begin{align*}
  \alpha(N,p) &= (4N^4p - 6N^3p^2 - 24N^3p) + (37N^2p^2 + 44N^2p - 72Np^2 - 4N^2 - 12Np + 36p^2)\\ 
              &\ge 2N^3p(2N-3p -12) + N(37Np^2 + 20Np - 4N - 12p) \\
              &\ge \frac{2}{5}N^4p + 43N^2p^2 \,,
\end{align*}
where $\alpha(N, p)$ is as in Theorem~\ref{thm:lower}.
Combining these results with Theorem~\ref{thm:lower} and $\frac{p-2}{p-1} \ge
\frac{1}{2}$, the desired inequality 
\[
  \omm^2 > \frac{1}{5N^2}\varphi(N)\log(N)
\]
follows.
\end{proof}

\newcommand{\etalchar}[1]{$^{#1}$}
\def\cprime{$'$}
\providecommand{\arxivref}[1]{\href{http://arxiv.org/abs/#1}{#1}}
\providecommand{\bysame}{\leavevmode\hbox to3em{\hrulefill}\thinspace}
\providecommand{\href}[2]{#2}

\end{document}